\documentclass[12pt,reqno]{amsart}
\usepackage{amsmath}
\usepackage{amssymb}
\usepackage{amstext}
\usepackage{a4wide}
\usepackage{graphicx}
\allowdisplaybreaks \numberwithin{equation}{section}
\usepackage{color}
\usepackage{cases}
\usepackage{hyperref}

\numberwithin{equation}{section}

\newtheorem{theorem}{Theorem}[section]
\newtheorem{proposition}[theorem]{Proposition}
\newtheorem{corollary}[theorem]{Corollary}
\newtheorem{lemma}[theorem]{Lemma}

\theoremstyle{definition}

\newtheorem{innercustomthm}{Theorem}
\newenvironment{customthm}[1]
{\renewcommand\theinnercustomthm{#1}\innercustomthm}
{\endinnercustomthm}

\newtheorem{definition}[theorem]{Definition}

\theoremstyle{remark}
\newtheorem{remark}[theorem]{Remark}

\begin{document}

\title[On helical vortex patches for 3D incompressible Euler equation]
{Structure of Green's function of elliptic equations and helical vortex patches for 3D incompressible Euler equations}

\author{Daomin Cao and Jie Wan}
	
\address{Institute of Applied Mathematics, Chinese Academy of Sciences, Beijing 100190, and University of Chinese Academy of Sciences, Beijing 100049,  P.R. China}
\email{dmcao@amt.ac.cn}
\address{School of Mathematics and Statistics, Beijing Institute of Technology, Beijing 100081,  P.R. China}
\email{wanjie@bit.edu.cn}


\begin{abstract}
We develop a new
structure of the Green's function of  a second-order elliptic operator in divergence form in a 2D bounded domain. Based on this structure and the theory of rearrangement of functions, we construct concentrated traveling-rotating helical vortex patches to 3D incompressible Euler equations in an infinite pipe. By solving an equation for vorticity
\begin{equation*}
w=\frac{1}{\varepsilon^2}f_\varepsilon\left(\mathcal{G}_{K_H}w-\frac{\alpha}{2}|x|^2|\ln\varepsilon|\right) \ \ \text{in}\ \Omega
\end{equation*}
for small $ \varepsilon>0 $ and considering a certain maximization problem  for the vorticity, where $ \mathcal{G}_{K_H} $ is the inverse of an elliptic operator $ \mathcal{L}_{K_H} $ in divergence form, we get the existence  of a family of concentrated helical vortex patches, which tend asymptotically to a singular helical vortex filament evolved by the binormal curvature flow. We also get nonlinear orbital stability of the  maximizers in the variational problem under $ L^p $ perturbation when $ p\geq 2. $\\


\end{abstract}
	
\maketitle

\section{Introduction and main results}

The incompressible Euler equation  confined in a 3D  domain $ D $ is governed by
\begin{equation}\label{Euler eq}
\begin{cases}
\partial_t\mathbf{v}+(\mathbf{v}\cdot \nabla)\mathbf{v}=-\nabla P,\ \ &D\times (0,T),\\
\nabla\cdot \mathbf{v}=0,\ \ &D\times (0,T),\\
\mathbf{v}(\cdot, 0)=\mathbf{v}_0(\cdot), \ \ &D,
\end{cases}
\end{equation}
where the domain $ D\subseteq \mathbb{R}^3 $, $ \mathbf{v}=(v_1,v_2,v_3) $ is the velocity field, $ P$ is the scalar pressure and  $ \mathbf{v}_0 $ is the initial velocity field.  If $ D $ has a boundary,  the following impermeable boundary condition
is usually assumed
\begin{equation*}
\mathbf{v}\cdot \mathbf{n}=0,\ \ \partial D\times (0,T),
\end{equation*}
where $ \mathbf{n} $ is the outward unit normal to $ \partial D $.

Define the associated vorticity field   $ \mathbf{w}=(w_1,w_2,w_3)=curl \mathbf{v}=\nabla\times \mathbf{v} $, which describes the rotation of the fluid.  Then $ \mathbf{w} $ satisfies the vorticity equations (see \cite{MB, MP})
\begin{equation}\label{Euler eq2}
\begin{cases}
\partial_t\mathbf{w}+(\mathbf{v}\cdot \nabla)\mathbf{w}=(\mathbf{w}\cdot \nabla)\mathbf{v},\ \ &D\times (0,T),\\
\mathbf{w}(\cdot, 0)=\nabla\times \mathbf{v}_0(\cdot), \ \ &D.
\end{cases}
\end{equation}


\subsection{Binormal curvature flow}

The study of Euler equation traced back to Helmholtz \cite{He} in 1858, who first considered the vorticity equations \eqref{Euler eq2} of the flow and found that the vortex rings, which are toroidal regions (vorticity region) in which the vorticity has small cross-section, translate with a constant speed alone the axis of symmetry. The translating speed of vortex rings was then studied by Kelvin and Hick \cite{Lamb} in 1899. To be more precise, let $ l $ be any oriented closed curve with tangent vector field $\mathbf{t}$ that encircles the afore mentioned vorticity region   once and $ \sigma $ is any surface with $ l $ as its boundary. Then the circulation of a vortex is defined by
\begin{equation}\label{1001}
d=\oint_l\mathbf{v}\cdot \mathbf{t}dl=\iint_{\sigma}\mathbf{w}\cdot\mathbf{n}d\sigma.
\end{equation}
It was shown in \cite{Lamb} that if the vortex ring has radius $ r^* $, circulation $ d $ and its cross-section $\varepsilon $ is small, then the vortex ring moves with the speed
\begin{equation}\label{1002}
\frac{d}{4\pi r^*}\left( \ln\frac{8r^*}{\varepsilon}-\frac{1}{4}\right).
\end{equation}

As for the motion of general vortex filament, Da Rios \cite{DR} in 1906, and Levi-Civita \cite{LC} in 1908, formally found the general law of motion of a vortex filament with a small cross-section of radius $ \varepsilon $ and a fixed circulation, uniformly distributed around an evolving curve $ \Gamma(t) $, which is well-known as the binormal curvature flow.
More precisely, if $ \Gamma(t) $ is parameterized as $ \gamma(s, t) $, where $ s $ is the  parameter of arclength, then $ \gamma(s, t) $ asymptotically obeys a law of the form (see \cite{LC2,Ric2})
\begin{equation}\label{1003}
\partial_t \gamma=\frac{d}{4\pi}|\ln\varepsilon|(\partial_s\gamma\times\partial_{ss}\gamma)=\frac{d\bar{K}}{4\pi}|\ln\varepsilon|\mathbf{b}_{\gamma(t)},
\end{equation}
where $ d $ is the circulation of the velocity field on the boundary of sections
to the filament, which is assumed to be a constant independent of $ \varepsilon $, $ \mathbf{b}_{\gamma(t)}  $ is the unit binormal and $ \bar{K} $ is its local curvature. If we rescale $ t = | \ln\varepsilon|^{-1}\tau $, then
\begin{equation}\label{1004}
\partial_\tau \gamma=\frac{d\bar{K}}{4\pi}\mathbf{b}_{\gamma(\tau)}.
\end{equation}
Hence, under the binormal curvature flow \eqref{1004} vortex filaments
move simply in the binormal direction with speed proportional to the local curvature and the circulation, see the survey papers by Ricca \cite{Ric,Ric2} for more detail. It is worthwhile to note that, when $ \Gamma $ is a circular filament, the leading term of \eqref{1002} coincides with the coefficient of right hand side of \eqref{1003} since the local curvature $ \bar{K}=\frac{1}{r^*}. $

From mathematical justification, Jerrad and Seis \cite{JS} first gave a precise form to Da Rios’ computation under some mild conditions on a solution to \eqref{Euler eq2} which remains suitably concentrated
around an evolving vortex filament. Their result shows that under some assumptions on a solution $ \mathbf{w}_\varepsilon $ of \eqref{Euler eq2}, there holds in the sense of distribution,
\begin{equation}\label{1005}
\mathbf{w}_\varepsilon(\cdot,|\ln\varepsilon|^{-1}\tau)\to d \delta_{\gamma(\tau)}\mathbf{t}_{\gamma(\tau)},\ \ \text{as}\ \varepsilon\to0,
\end{equation}
 where $ \gamma(\tau) $ satisfies \eqref{1004}, $ \mathbf{t}_{\gamma(\tau)} $ is the tangent unit vector of $ \gamma $ and $ \delta_{\gamma(\tau)} $ is the uniform Dirac measure on the curve. For more results of this problem, see  \cite{JS2}.

The existence of a family of solutions to \eqref{Euler eq2} satisfying \eqref{1005}, where $ \gamma(\tau) $ is a given curve  evolved by the binormal flow \eqref{1004}, is well-known as the vortex filament conjecture, which is unsolved except for the filament being several kinds of special curves: the straight lines, the traveling circles and the traveling-rotating helices.
For the problem of vortex concentrating near straight lines, it corresponds to the planar Euler equations concentrating near a collection of given points governed by the 2D point vortex model, see \cite{CF,CLW,DDMW,Li,SV} for example. When the filament is a traveling circle with radius $ r^* $, by \eqref{1004} the curve is
\begin{equation}\label{1006}
\gamma(s,\tau)=\left( r^*\cos\left( \frac{s}{r^*}\right), r^*\sin\left( \frac{s}{r^*}\right), \frac{d}{4\pi r^*}\tau\right)^t,
\end{equation}
where $ \mathbf{v}^t $ is the transposition of a vector $ \mathbf{v} $. Fraenkel \cite{Fr1} first gave a construction of  vortex rings with small cross-section without change of form concentrating near a traveling circle  \eqref{1006}, see \cite{ALW, Bur, De2, DV, FB} for more results.

For vortex filament being a helix satisfying \eqref{1004},  the curve is parameterized as
\begin{equation}\label{1007}
\gamma(s,\tau)=\left( r_*\cos\left( \frac{-s-a_1\tau}{\sqrt{k^2+r_*^2}}\right),  r_*\sin\left( \frac{-s-a_1\tau}{\sqrt{k^2+r_*^2}}\right), \frac{ks-b_1\tau}{\sqrt{k^2+r_*^2}}\right)^t,
\end{equation}
where $ r_* > 0, k\neq 0 $ are constants characterizing  the distance between a point in $ \gamma(\tau) $ and the $ x_3 $-axis and the pitch of the helix, and
\begin{equation*}
a_1=\frac{dk}{4\pi (k^2+r_*^2)}, \ b_1=\frac{dr_*^2}{4\pi (k^2+r_*^2)}.
\end{equation*}
Observe that the local curvature and torsion of the helix are $ \frac{r_*}{k^2+r_*^2} $ and $ \frac{k}{k^2+r_*^2} $ respectively, and the parametrization \eqref{1007} satisfies \eqref{1004}. It   should be noted that the curve parameterized by \eqref{1007} is a traveling-rotating helix. Let us define  for any $ \theta\in[0,2\pi] $
\begin{equation*}
\bar{R}_\theta=\begin{pmatrix}
\cos\theta & \sin\theta  \\
-\sin\theta &\cos\theta
\end{pmatrix}, \  \
\bar{Q}_\theta=\begin{pmatrix}
\bar{R}_\theta & 0  \\
0 & 1
\end{pmatrix}.
\end{equation*}
By direct calculations it is easy to see that
\begin{equation*}
\gamma(s,\tau)=\bar{Q}_{\frac{a_1\tau}{\sqrt{k^2+r_*^2}}}\gamma(s,0)+\left( 0,0,-\frac{b_1\tau}{\sqrt{k^2+r_*^2}}\right)^t.
\end{equation*}
\eqref{1007} with $ k>0 $ and $ k<0 $ correspond to the left-handed helix and the right-handed helix  respectively. We always consider the  case of $ k>0 $ in this paper.

\subsection{2D vorticity equation}
Our purpose in this paper is to study the vortex filament conjecture in a special case that the filament is a helix. We want to construct a family of 3D Euler flow
such that the associated vorticity field $ \mathbf{w}_\varepsilon $ concentrates near a traveling-rotating helical vortex filament \eqref{1007}.
Let us first  define solutions of Euler equation with helical symmetry and deduce it to a 2D vorticity equation, see for example \cite{Du,ET}. For $ k>0 $, define a group of one-parameter helical transformations $ \mathcal{H}_k=\{H_{\bar{\rho}}:\mathbb{R}^3\to\mathbb{R}^3\mid  \bar{\rho}\in\mathbb{R}\} $, where
\begin{equation*}
H_{\bar{\rho}}(x_1,x_2,x_3)^t=(x_1\cos\bar{\rho}+x_2\sin\bar{\rho}, -x_1\sin\bar{\rho}+x_2\cos\bar{\rho}, x_3+k\bar{\rho})^t.
\end{equation*}
Define  the field of tangents of symmetry lines of $ \mathcal{H}_k $
\begin{equation*}
\overrightarrow{\zeta}= (x_2, -x_1, k)^t.
\end{equation*}

Suppose that the domain $ D  $ is    helical, that is,  $ H_{\bar{\rho}}(D)=D $ for any $\bar{\rho}$.
Let $ \Omega=D\cap\{x\mid x_3=0\} $ be the section of $ D $ over $ x_1Ox_2 $ plane. Then $ D  $ can be generated by $ \Omega $ by  letting $ D=\cup_{\bar{\rho}\in\mathbb{R}}H_{\bar{\rho}}(\Omega) $. Hereafter we suppose that $ \Omega $ is simply-connected bounded domain with $ C^\infty $ boundary.

Now one can define helical functions and vector fields in a helical domain $ D $. A scalar function $ h $ is  $ helical $, if $ h(H_{\bar{\rho}}(x))=h(x) $ for any $ \bar{\rho}\in\mathbb{R}, x\in D. $
A vector field $ \mathbf{h}=(h_1,h_2,h_3) $ is $ helical $, if $ \mathbf{h}(H_{\bar{\rho}}(x))=\bar{Q}_{\bar{\rho}} \mathbf{h}(x) $ for any $ \bar{\rho}\in\mathbb{R}, x\in D. $
Helical solutions of \eqref{Euler eq} are then defined as follows.
\begin{definition}
	A function pair ($\mathbf{v}, P$) is called a $ helical$ solution  of \eqref{Euler eq} in $ D $, if ($\mathbf{v}, P$) satisfies \eqref{Euler eq} and both vector field $ \mathbf{v} $ and scalar function $ P $ are helical.
\end{definition}
Similar to \cite{ET}, we assume that helical solutions satisfy  the \textit{orthogonality condition}:
\begin{equation}\label{ortho}
\mathbf{v}\cdot \overrightarrow{\zeta}=0.
\end{equation}
\eqref{ortho} is also called the no-swirl condition. From \cite{ET} we know that, under the assumption \eqref{ortho},  the vorticity field $ \mathbf{w} $ satisfies
\begin{equation*}\label{w formula}
\mathbf{w}=\frac{w_3}{k}\overrightarrow{\zeta},
\end{equation*}
where $ w_3=\partial_{x_1}v_2-\partial_{x_2}v_1 $ is the third component of vorticity field $ \mathbf{w} $, which is a helical function.
Moreover, if define $ w(x_1,x_2,t)=w_3(x_1,x_2,0,t) $, then $ w $ satisfies the 2D vorticity-stream equations

\begin{equation}\label{vor str eq}
\begin{cases}
\partial_t w+\nabla^\perp\varphi\cdot \nabla w=0,\ \ &\ \Omega\times(0,T),\\
w=\mathcal{L}_{K_H}\varphi,\ \ & \ \Omega\times(0,T),\\
\varphi=0,\ \ & \ \partial \Omega\times(0,T),\\
w(\cdot,0)=w_0(\cdot),\ \ & \ \Omega,
\end{cases}
\end{equation}
where $ \varphi $ is called the stream function,  $ \mathcal{L}_{K_H}\varphi:=-\text{div}(K_H(x_1,x_2)\nabla\varphi) $ is a second order elliptic operator of divergence type with the coefficient matrix
\begin{equation}\label{coef matrix}
K_H(x_1,x_2)=\frac{1}{k^2+x_1^2+x_2^2}
\begin{pmatrix}
k^2+x_2^2 & -x_1x_2 \\
-x_1x_2 &  k^2+x_1^2
\end{pmatrix},
\end{equation}
and $ \perp $ denotes the clockwise rotation through $ \pi/2 $, i.e., $ (a,b)^\perp=(b,-a) $. Note that $ K_H $ is strictly positive-definite, by the elliptic regularity theory, for any $ p\in(1,+\infty) $  one can define a continuous linear operator $ \mathcal{G}_{K_H} $ from $ L^p(\Omega) $ to $ W^{2,p}\cap W^{1,p}_0(\Omega) $ such that $ \varphi=\mathcal{G}_{K_H}w $ for any $ w\in L^p(\Omega)  $ (see also Proposition \ref{exist of G_K} in section 2). Thus \eqref{vor str eq} can be rewritten as the following vorticity equations
\begin{equation}\label{vor str equ2}
\begin{cases}
\partial_t w+\nabla^\perp\mathcal{G}_{K_H}w\cdot \nabla w=0,\ \ & \ \Omega\times(0,T),\\
w(\cdot,0)=w_0(\cdot),\ \ & \ \Omega.
\end{cases}
\end{equation}
If $ p>4/3 $, then $ \nabla^\perp\mathcal{G}_{K_H}w\in W^{1,p}(\Omega)\subset L^{p'}(\Omega) $, where $ p' $ is the conjugate exponent of p. Thus we can give the definition of weak solutions to the vorticity equation \eqref{vor str equ2} since $  w\cdot\nabla^\perp\mathcal{G}_{K_H}w $ is integrable.
\begin{definition}
	Suppose $p\in(4/3,+\infty]$. We call $w(x,t)\in L^\infty((0,+\infty);L^p(\Omega))$ a weak solution to \eqref{vor str equ2} if	for all $\xi\in C_c^{\infty}(\Omega\times[0,+\infty))$
	\begin{equation}\label{997}
	\int_\Omega w_0(x)\xi(x,0)dx+\int_0^{+\infty}\int_\Omega w(\partial_t\xi+\nabla\xi\cdot \nabla^\perp \mathcal{G}_{K_H}w)dxdt=0.
	\end{equation}

\end{definition}

when $p=+\infty$, the existence and uniqueness result for weak solution to vorticity equation \eqref{vor str equ2} was firstly proved by  \cite{ET}. For general $p>4/3$, by using an approximation procedure and the Diperna-Lions theory of linear transport equations, \cite{Ben} proved the following existence result.

\begin{customthm}{A}\label{A}
 Let $ \Omega $ be a simply-connected bounded domain with $ C^\infty $ boundary.	Suppose $4/3<p<+\infty$ and $w_0\in L^p(\Omega)$. Then there exists a weak solution $w(x,t)\in L^\infty((0,+\infty);L^p(\Omega))$ to the vorticity equation \eqref{vor str equ2}. Moreover, if $ p\geq2 $, then
	\begin{itemize}
		\item[(i)] all $L^\infty( (0,+\infty); L^p(\Omega))$ solutions belong to $C( [0,+\infty); L^p(\Omega))$;
		\item[(ii)] for any weak solution $w(x,t)\in L^{\infty}((0,+\infty);L^p(\Omega))$,  we have $w(x,t)\in \textbf{R}_{w_0}$ for all $t\geq 0$, where $\textbf{R}_{w_0}$ denotes the rearrangement class of $w_0$,
		\begin{equation*}
		\textbf{R}_{w_0}:=\{ v \in L^1_{loc}(\Omega)\mid  |\{v>a\}|=|\{w_0>a\}| ,\forall a\in \mathbb R\}.
		\end{equation*}
	\end{itemize}
\end{customthm}
A typical weak solution  for characterizing an isolated region of vorticity with jump discontinuity is the vortex patch solution, that is, the initial vorticity has the form
\begin{equation}\label{patch}
w_0(x)=\frac{1}{\varepsilon^2} \textbf{1}_{A_0},
\end{equation}
where $ \textbf{1}_{A_0} $ the characteristic function of $ {A_0} $, namely $ \textbf{1}_{A_0}=1 $ in $ {A_0} $ and $ \textbf{1}_{A_0}=0 $ in $ {A_0}^c  $,  and $\varepsilon$ is the vorticity strength parameter. Clearly from Theorem \ref{A}, there exists the unique solution $ w(x,t) $ of \eqref{vor str equ2} being of the form $ w(\cdot,t)=\textbf{1}_{A_t} $, where $ |A_0|=|A_t| $ for any $ t>0 $. More results of global well-posedness of solutions to Euler equation \eqref{Euler eq} with helical symmetry can be seen in \cite{AS,BLN,Du,Jiu} for instance.

Observe that the motion of  the intersection point of the helix \eqref{1007} and the $ x_1Ox_2 $ plane is a clockwise rotation about the origin with constant angular velocity, so it suffices to construct  a family of rotation-invariant solutions to \eqref{vor str equ2}, such that the associated helical vorticity field $ \mathbf{w}_\varepsilon $ concentrates near the curve \eqref{1007}.

Thereafter we assume that $ D $ is an infinite pipe: for any $ R^*>0 $, define $ D= B_{R^*}(0)\times \mathbb{R}=\{(x_1,x_2,x_3)\mid (x_1,x_2)\in B_{R^*}(0), x_3\in \mathbb{R}\} $. So $ \Omega=B_{R^*}(0) $. Let $ \alpha $ be a constant. We look for solutions of \eqref{vor str equ2} being of the form
\begin{equation}\label{104}
w(x',t)=W(\bar{R}_{-\alpha|\ln\varepsilon| t}(x')),
\end{equation}
where $ x'=(x_1,x_2)\in B_{R^*}(0)$. Then one computes directly  that $ W  $ satisfies
\begin{equation}\label{rot eq}
\begin{split}
\nabla W\cdot \nabla^\perp \left( \mathcal{G}_{K_H}W-\frac{\alpha}{2}|x'|^2|\ln\varepsilon|\right) =0.
\end{split}
\end{equation}
So formally if
\begin{equation}\label{rot eq2}
\begin{split}
W=f_\varepsilon\left(\mathcal{G}_{K_H}W-\frac{\alpha}{2}|x'|^2|\ln\varepsilon|\right), \ \ \ B_{R^*}(0),
\end{split}
\end{equation}
for some function $ f_\varepsilon $, then \eqref{rot eq} automatically holds.
For a solution $ W $ of \eqref{rot eq2}, one can get a rotating-invariant solution  $w$ of \eqref{vor str equ2} with angular velocity $ \alpha|\ln\varepsilon| $ by simply using  \eqref{104}. Observe that if we define $ \varPhi=\mathcal{G}_{K_H}W $, then it is equivalent to solve a semilinear elliptic problem in divergence form
\begin{equation*}
\begin{cases}
-\text{div}(K_H(x)\nabla  \varPhi)=f_\varepsilon\left(\varPhi-\frac{\alpha}{2}|x'|^2|\ln\varepsilon|\right), \ \ & \ B_{R^*}(0),\\
\varPhi=0,\ & \ \partial B_{R^*}(0),
\end{cases}
\end{equation*}

For a helix  \eqref{1007}, there are a few results of existence of true solutions of \eqref{Euler eq2} concentrating on this curve.  D$\acute{\text{a}}$vila et al. \cite{DDMW2} considered smooth helical Euler flows concentrating near a single  helix  in the whole space $ \mathbb{R}^3  $  by considering a Liouville type equation
\begin{equation*}
\begin{split}
-\text{div}(K_H(x)\nabla u)=f_\varepsilon\left( u-\alpha|\ln\varepsilon|\frac{|x|^2}{2}\right) \ \ \text{in}\ \ \mathbb{R}^2,
\end{split}
\end{equation*}
where  $ f_\varepsilon(t)=\varepsilon^2e^t $ and $ \alpha $ is chosen properly. Note that by the choice of  $ f_\varepsilon $, the support set of vorticity maybe still the whole plane, rather than a compact set with small diameter.

\subsection{Main results}
Our main results are divided into two parts: the expansion of Green's function of a second-order elliptic operator with Dirichlet condition, and the existence and orbital stability of concentrated helical vortex patches to 3D Euler equations in infinite pipes.

It follows from \eqref{rot eq2} that the Green's function of $ \mathcal{L}_{K_H} $ appearing
in the construction of the stream function  plays an important role in constructing concentrated helical vorticity field to Euler equations \eqref{Euler eq}. Thus we shall need to understand the structure of Green's function of  the elliptic operator $ \mathcal{L}_{K_H} $.  To this end, we first study  Green's function of a linear problem
\begin{equation}\label{diver form}
\begin{cases}
\mathcal{L}_Ku:=-\text{div}(K(x)\nabla u)= f,\ \ &x\in U,\\
u=0,\ \ &x\in\partial U,
\end{cases}
\end{equation}
where $ U\subset \mathbb{R}^2 $ is a simply-connected bounded domain with smooth boundary, $ K=(K_{i,j})_{2\times2} $ is a positive-definite matrix satisfying
\begin{enumerate}
	\item[($\mathcal{K}$1).]  $ K_{i,j}(x)\in C^{\infty}(\overline{U}) $ for $ 1\le i,j\le 2. $
	\item[($\mathcal{K}$2).] $ -\text{div}(K(x)\nabla \cdot) $ is uniformly elliptic, that is, there exist  $ \Lambda_1,\Lambda_2>0 $ such that $$ \Lambda_1|\zeta|^2\le (K(x)\zeta|\zeta) \le \Lambda_2|\zeta|^2,\ \ \ \ \forall\ x\in U, \ \zeta\in \mathbb{R}^2.$$
\end{enumerate}
Since $ K $ satisfies  $(\mathcal{K}1)$--$(\mathcal{K}2)$, one can define a Green's operator $ \mathcal{G}_K  $ such that for every $ f\in L^{p}(U) $ with $ p>1 $, the function $ u=\mathcal{G}_Kf $ is a weak solution to the problem \eqref{diver form}, see Proposition \ref{exist of G_K}.

Let $ \Gamma(x)=-\frac{1}{2\pi}\ln|x|$ be the fundamental solution of the Laplacian $ -\Delta $ in $ \mathbb{R}^2 $ and $ T_x $ be a $ C^\infty $ positive-definite matrix-valued function determined by $ K $ satisfying
\begin{equation}\label{matrix T1}
T_x^{-1}(T_x^{-1})^{t}=K(x)\ \ \ \ \forall\ x\in U.
\end{equation}
Clearly  by the Cholsky decomposition and the properties of positive-definite matrix, such $ T_x $ exists and is unique.


Our first result shows a structure of the Green's function of \eqref{diver form}   as follows.
\begin{theorem}\label{expe of G_K}
	Let $ q>2. $ There exists a function $ S_K\in C^{0,\gamma}_{loc} (U\times U) $ for some $ \gamma\in(0,1) $ such that for every $f \in L^{q}(U)$ and every $ x\in U $,
	\begin{equation}\label{2-11}
	\begin{split}
	\mathcal{G}_Kf(x)=\int_{U}\left (\frac{\sqrt{\det K(x)}^{-1}+\sqrt{\det K(y)}^{-1}}{2}\Gamma\left (\frac{T_x+T_y}{2}(x-y)\right )+S_K(x,y)\right)f(y)dy.
	\end{split}
	\end{equation}
	Moreover, 
	\begin{equation}\label{2-111}
	S_K(x,y)=S_K(y,x)  \ \ \ \ \forall\ x,y \in U,
	\end{equation}
and there is a positive constant $C$ such that
	\begin{equation}\label{2-12}
	S_K(x,y)\leq C  \ \ \ \ \forall\ x,y \in U.
	\end{equation}
	
\end{theorem}

Theorem \ref{expe of G_K} implies that the  Dirichlet problem in divergence form \eqref{diver form} has a
Green's  function $ G_K : U\times U \to \mathbb{R}$  defined for each $ x,y\in U $ with $ x\neq y $ by
\begin{equation}\label{Green}
G_K(x, y)=\frac{\sqrt{\det K(x)}^{-1}+\sqrt{\det K(y)}^{-1}}{2}\Gamma\left (\frac{T_x+T_y}{2}(x-y)\right )+S_K(x,y),
\end{equation}
which is not known in the existing literature. It should be noted that when choosing the matrix $ K(x)=Id $ and $ \frac{1}{b(x)}Id $, which correspond to 2D Euler flows and lake equations respectively, \eqref{Green} coincides with the classical results, see \cite{CLW,De2}.

\begin{remark}
	Below we give examples of the matrix $ K $ in \eqref{diver form} to explain the expansion of Green's function in Theorem \ref{expe of G_K}.
	
\noindent$ Example $ 1. If $ K(x)=Id $, then \eqref{diver form} is the standard Laplacian problem, which corresponds to the vorticity-stream formulation to 2D Euler equations, see \cite{CF,CLW,DV,T} for example. By Theorem \ref{expe of G_K},  Green's function becomes
	\begin{equation*}
	G_1(x,y)=\Gamma(x-y)+S_1(x,y),\ \ \ \ \forall\ x,y\in U.
	\end{equation*}
	So in this case, $ S_1(x,y)=-H(x,y) $, where $ H(x,y) $ is the regular part of Green's function of $ -\Delta $ in $ U $ with zero-Dirichlet data.
	
\noindent$ Example $ 2. If  $ K(x)=\frac{1}{b(x)}Id $, where $ b\in C^1(U) $ and $ \inf_{U}b>0 $, then \eqref{diver form}  corresponds to vorticity-stream formulation to the 2D lake equations, see \cite{De2} for example. It is not hard to get that $ \det K=\frac{1}{b^2} $ and $ T=\sqrt{b}Id $. From Theorem \ref{expe of G_K},  Green's function becomes
	\begin{equation*}
	G_b(x,y)=\frac{b(x)+b(y)}{2}\Gamma\left (\frac{\sqrt{b(x)}+\sqrt{b(y)}}{2}(x-y)\right )+S_b(x,y) \ \ \ \ \forall\ x,y\in U,
	\end{equation*}
	which coincides with results in \cite{De2}.
\end{remark}

\begin{remark}
There are many articles showing the existence, uniqueness and $ L^p$-estimate of the Green's function $ G_K $ of \eqref{diver form}, see   \cite{GT,GW82,KN85,LSW63,TKB13} for example. However, few results give an explicit decomposition of $ G_K $. To the best of our knowledge, the structure of $ G_K $ in Theorem \ref{expe of G_K} is new and we believe that it is of independent interest.

\end{remark}


The structure of Green's function of the Dirichlet problem \eqref{diver form} shown in Theorem \ref{expe of G_K} can be used to construct concentrated helical vorticity field to 3D incompressible Euler equations in infinite pipes. Our second result is  the existence of concentrated helical vortex patches with small cross-section in $ B_{R^*}(0)\times \mathbb{R} $, which tend asymptotically to a single left-handed helix  \eqref{1007}. In the sequel, we always assume $ \Omega=B_{R^*}(0) $.

\begin{theorem}\label{thm01}
$\mathrm{[Existence]}$ Let $ k>0 $, $ d>0 $ and $ r_*\in (0,R^*) $ be any given numbers. Let  $ \gamma(\tau) $ be the helix parameterized by equation \eqref{1007}. Then for any $ \varepsilon\in (0,\min\{1, \sqrt{|\Omega|/d}\})$, there exists
a solution pair $ (\mathbf{v}_\varepsilon, P_\varepsilon)(x,t)  $ of  \eqref{Euler eq} such that the support set of  $ \mathbf{w}_\varepsilon $ is a topological traveling-rotating helical tube that does not change form and concentrates near $ \gamma(\tau) $ in sense of \eqref{1005}, that is for all $ \tau $,
\begin{equation*}
\mathbf{w}_\varepsilon(\cdot,|\ln\varepsilon|^{-1}\tau)\to d \delta_{\gamma(\tau)}\mathbf{t}_{\gamma(\tau)},\ \ \text{as}\ \varepsilon\to0.
\end{equation*}
Moreover, the following properties hold:
\begin{enumerate}
\item Define $\omega_\varepsilon(x_1,x_2,t)$ the third component of $\mathbf{w}_\varepsilon(x_1,x_2,0,t) $. Then
\begin{equation*}
\omega_{\varepsilon}=\frac{1}{\varepsilon^2}\mathbf{1}_{\{\mathcal{G}_{K_H}\omega_{\varepsilon}-\frac{\alpha |x|^2}{2}\ln\frac{1}{\varepsilon} -\mu^{\varepsilon}>0\}},
\end{equation*}
where $ \alpha=\frac{d}{4\pi k\sqrt{k^2+r_*^2}}  $, and $\mu^{\varepsilon} $ is a  Lagrange multiplier.
\item Define $ \bar{A}_\varepsilon=supp(\omega_\varepsilon) $ the cross-section of $ \mathbf{w}_\varepsilon $. Then there are $ r_1,r_2>0 $ such that
\begin{equation*}
r_1\varepsilon\leq diam(\bar{A}_\varepsilon)\leq r_2\varepsilon.
\end{equation*}
\item $ \mathbf{v}_\varepsilon\cdot \mathbf{n}=0$ on $\partial B_{R^*}(0)\times \mathbb{R}. $
\end{enumerate}
\end{theorem}

Our strategy to show  Theorem \ref{thm01} is to construct a family of vortex patch solutions to  \eqref{rot eq2} with $ f_{\varepsilon}(t)=\frac{1}{\varepsilon^2}\textbf{1}_{\{t>\mu^\varepsilon\}} $ for some $ \mu^\varepsilon $.
To be more clear, let us introduce some definition first.
Let  $ \mathcal{E}(\omega) $ and $ \mathcal{I}(\omega) $ be the kinetic energy and the moment of inertia defined respectively by
\begin{equation*}
	\mathcal{E}(\omega)=\frac{1}{2}\int_{\Omega}\omega\mathcal{G}_{K_H}\omega dx,\ \ \ \
	\mathcal{I}(\omega)=\frac{1}{2}\int_{\Omega}|x|^2\omega  dx,
\end{equation*}
see \eqref{KE}, \eqref{MI} and \eqref{Energy}  in section 3 for more details.
 Consider the maximization of the following functional
\begin{equation*}
\begin{split}
\mathcal{E}_\varepsilon(\omega)=\mathcal{E}(\omega)-\alpha\ln\frac{1}{\varepsilon}\mathcal{I}(\omega)
=\frac{1}{2}\int_{\Omega}\omega\mathcal{G}_{K_H}\omega dx-\frac{\alpha }{2}\ln\frac{1}{\varepsilon}\int_{\Omega}|x|^2\omega  dx
\end{split}
\end{equation*}
over the constraint set
\begin{equation*}
\mathcal{M}_{\varepsilon}=\left \{\omega\in L^\infty(\Omega)~|~\int_{\Omega}\omega dx = d,\ \text{and} ~ 0\le \omega \le \frac{1}{\varepsilon^2} \right \}.
\end{equation*}
Here $ \alpha>0 $ is a constant to be determined properly. It is not hard to get the existence and profile of maximizers, which are vortex patch solutions to \eqref{rot eq2}, see Proposition \ref{exist of max}. Then we give asymptotic estimates of maximizers. It is worth noting that,  asymptotic estimates of maximizers rely on the structure of Green's function of the operator $ \mathcal{L}_{K_H} $, which is shown in Theorem  \ref{expe of G_K}.
Thus, using the characterization that a maximizer  has the largest energy and the theory of  rearrangement functions, we get the asymptotic estimates  of  solutions.
We observe that the limiting location of maximizers  is a maximum of $ Y $ defined by \eqref{def of Y}, which is $ (r_*,0) $ up to rotation. Finally,   we prove that maximizers correspond to a family of concentrated helical vortex patch vorticity fields to Euler equations \eqref{Euler eq2},  which tends asymptotically to a helix  \eqref{1007} in sense of \eqref{1005}.

\begin{remark}
The condition $  \varepsilon\in (0,\min\{1, \sqrt{|\Omega|/d}\}) $ follows from the fact that $ d=\int_\Omega \omega dx\leq \frac{|\Omega|}{\varepsilon^2} $ and $ \ln\frac{1}{\varepsilon}>0 $. So Theorem \ref{thm01} shows the existence of concentrated helical Euler flows  not only for small $ \varepsilon $, but for all $ \varepsilon<\min\{1, \sqrt{|\Omega|/d}\}. $

Note that the elliptic operator in 2D Euler equation is $ -\Delta $, while in 3D axisymmetric case the operator is $ -\frac{1}{b(x)}\nabla\cdot(b(x)\nabla) $ with $ b(x)=r $, see \cite{DV,FB,SV} for example. In contrast to the 2D and 3D axisymmetric problem, the associated operator $ \mathcal{L}_{K_H} $ in vorticity equations \eqref{vor str eq} is an elliptic operator in divergence form, which brings essential difficulty in studying the existence and asymptotic behavior of solutions. It is hard to reduce the operator $ \mathcal{L}_{K_H} $ to the standard Laplacian  by means of a single change of coordinates. Moreover, since the eigenvalues of $ K_H $ are not the same, the limiting shape of support set of solutions is not a disc but an ellipse, which is different from the 2D and 3D axisymmetric cases.
\end{remark}

\begin{remark}
	Indeed, the existence of solutions of general elliptic  equations in divergence form has been studied by \cite{PS}, who considered a singularly perturbed elliptic problem:
	\begin{equation}\label{PS eq}
	\begin{split}
	-\varepsilon^2 \text{div}(K(x)\nabla u) +V(x)u=u^{p},\ \ \ \ x\in \mathbb{R}^n,
	\end{split}
	\end{equation}
	where $ K(x) $ is strictly positive definite, $ n\geq 3 $, $ p\in (1,\frac{n+2}{n-2}) $ and $ V\in C^{1}(\mathbb{R}^n) $ is positive. The authors constructed solutions concentrating near minimizers of $ V(x)^{\frac{p+1}{p-1}}\sqrt{det(K(x))} $ by the penalization technique. However, it seems that the method in \cite{PS} can not be used in our situation since it depends on the positiveness of $ V $.
	
\end{remark}

\begin{remark}
Let us point out that our method of proving  Theorem \ref{thm01} is different from that used in  \cite{ALW},\cite{DDMW2}, where existence of solutions was obtained by studying the equation satisfied by stream function by Lyapunov-Schmidt reduction. Instead, we study the equation satisfied by the vorticity, that is \eqref{rot eq} or \eqref{rot eq2}. The solutions obtained in this way have more information. They maximize
the energy, for instance. Such characteristics is very useful in establishing orbital stability, which will be given next. Additionally, by the nonlinearity we choose, we obtain the so-called vortex patches, which are of different type of solutions  from those obtained in \cite{DDMW2} as well.
\end{remark}

Before stating our third result concerning the nonlinear stability of concentrated vortex patches constructed in Theorem \ref{thm01}, let us define the set of maximizers
\begin{equation}\label{s}
\mathcal{S}_\varepsilon:=\{\omega\in \mathcal{M}_\varepsilon \mid \mathcal{E}_\varepsilon(\omega)=\sup_{\mathcal{M}_\varepsilon}\mathcal{E}_\varepsilon\}.
\end{equation}
According to Proposition \ref{exist of max}, $\mathcal{S}_\varepsilon$ is not empty, and any element in $\mathcal{S}_\varepsilon$ is a rotation-invariant vortex patch  to \eqref{vor str equ2}, which corresponds to a 3D Euler flow with helical symmetry.  Observe that by the rotational invariance of energy $ \mathcal{E}_\varepsilon $,  for any $ \omega\in \mathcal{S}_\varepsilon $, $ \bar{R}_{\theta}(\omega)\in \mathcal{S}_\varepsilon $ for any $ \theta\in\mathbb{R} $.
Thus  the maximizer set $ \mathcal{S}_\varepsilon $ is invariant in the sense that if $w_0\in\mathcal{S}_\varepsilon $, then $w_t\in\mathcal{S}_\varepsilon$ for all $t>0$, where $\omega_t$ is a weak solution to the vorticity equation \eqref{vor str equ2} with initial $w_0$.   A  nature question is the stability of the maximizers, namely for any given initial vorticity $w_0$ that is sufficiently close to the set of rotation-invariant vortex patch solutions $\mathcal{S}_\varepsilon$ in some norm, will it be close to $\mathcal{S}_\varepsilon$ for all $t>0$ in the same norm? If it is true, we say $\mathcal{S}_\varepsilon$ is nonlinear orbitally stable.

There are many articles showing the stability of 2D and 3D axisymmetric Euler flows in the past few decades, see \cite{B5,B6,Ta,WP} for example.
A very effective method to prove nonlinear stability for smooth planar Euler flows is established by Arnold \cite{A2}, which was later extended to non-smooth flows, for example, vortex patches, see \cite{B5,B6,WP} in 2D and 3D axisymmetric cases. As for stability of helical Euler flows, \cite{Ben} showed a stability result for stationary smooth helical Euler flows in the $ L^2 $ norm of the velocity and vorticity by using the direct method of Lyapunov. Combining  the characterization of energy maximizers and the conservation of the kinetic energy and the moment of inertia for solutions to \eqref{vor str equ2}, we get the orbital stability to $\mathcal{S}_\varepsilon $ in the $ L^p $ sense for any $ p\geq 2 $.

\begin{theorem}\label{os}
$\mathrm{[Orbital~stability]}$ Let $2\leq p<+\infty$, $\varepsilon \in(0,\min\{1, \sqrt{|\Omega|/d}\} )$, and $\mathcal{S}_\varepsilon$ be defined by \eqref{s}. Then $\mathcal{S}_\varepsilon$ is orbitally stable in $L^p$ norm, or equivalently, for any $\rho>0$, there exists a $\delta>0$, such that for any $\omega_0\in L^p(\Omega)$ satisfying
\begin{equation*}
\inf_{\omega\in\mathcal{S}_\varepsilon}\|\omega_0-\omega\|_{L^p(\Omega)}<\delta,
\end{equation*}	
we have
\begin{equation*}
\inf_{\omega\in\mathcal{S}_\varepsilon}\|\omega_t-\omega\|_{L^p(\Omega)}<\rho
\end{equation*}
for all $t>0$, where $\omega_t$ is a weak solution to the vorticity equation \eqref{vor str equ2} with initial vorticity $\omega_0$.
\end{theorem}

\begin{remark}
In \cite{Ben}, under the assumption that
\begin{equation}\label{Ben cond}
0\leq -\frac{\nabla \mathcal{G}_{K_H}w(x)}{\nabla w(x)}\leq C, \ \ \forall\ x\in\Omega,
\end{equation}
the author obtained the nonlinear stability of smooth steady solutions to vorticity equation \eqref{vor str equ2} by constructing a certain first integral. However, for  many weak solutions to \eqref{vor str equ2} like vortex patches, \eqref{Ben cond} does not hold. In contrast to \cite{Ben}, we use the  characterization of energy maximizers to get the orbital stability of vortex patches  constructed in Theorem \ref{thm01}, which is not known before. Another interesting problem is  whether these solutions are stable. To the best of our knowledge, it is still unknown.
\end{remark}

The paper is organized as follows. In section 2, we give the expansion of Green's  function $ G_K $ of the Dirichlet problem \eqref{diver form} by a constructive way. In section 3, we prove the existence and profile of maximizers of $ \mathcal{E}_\varepsilon $ in the set $ \mathcal{M}_\varepsilon $, which corresponds to a rotation-invariant vortex patch to the vorticity equation \eqref{vor str equ2}. In section 4 and section 5, we prove the asymptotic behavior of maximizers and finish the proof of Theorem \ref{thm01}. By showing a compactness result and conservation of the energy and the moment of inertia, we prove the orbital stability of solutions constructed in Theorem \ref{thm01}  in section 6.

\section{Construction of Green's function}
In this section, we prove that Green's function of the linear problem \eqref{diver form} has a structure shown in Theorem \ref{expe of G_K}.
First by the  $ L^p $ theory for elliptic equations, one has the classical result:
\begin{proposition}\label{exist of G_K}
For every $ p\in(1,+\infty) $, there exists a linear continuous operator
$ \mathcal{G}_K: L^p(U) \to W^{2,p}(U) $ such that for every $ f\in L^p(U) $, the function $ u=\mathcal{G}_Kf $ is a weak solution of the problem \eqref{diver form}.
\end{proposition}

\begin{proof}
	See theorem 9.15 in  \cite{GT}, for example.
\end{proof}

\begin{proof}[Proof of Theorem \ref{expe of G_K}]
In the sequel, we give proof of Theorem \ref{expe of G_K}.  The idea is to represent the Green's  function of the Dirichlet
problem \eqref{diver form} as a combination of the fundamental solution of the classical Laplacian in $ \mathbb{R}^2 $ and some positive-definite matrices and functions on the domain $ U $.

We first determine $ S_K $. Define
	\begin{equation*}
	G_0(x,y):=\frac{\sqrt{\det K(x)}^{-1}+\sqrt{\det K(y)}^{-1}}{2}\Gamma\left (\frac{T_x+T_y}{2}(x-y)\right ).
	\end{equation*}
Then for fixed $ y\in U $,
	\begin{equation*}
	\begin{split}
	-\nabla_x\cdot(K(x)\nabla_x G_0(x,y))=&-\partial_{x_1}(K_{11}(x)\partial_{x_1}G_0(x,y)+K_{12}(x)\partial_{x_2}G_0(x,y))\\
	&-\partial_{x_2}(K_{21}(x)\partial_{x_1}G_0(x,y)+K_{22}(x)\partial_{x_2}G_0(x,y))\\
	 =&-K_{11}(x)\partial_{x_1x_1}G_0(x,y)-K_{12}(x)\partial_{x_1x_2}G_0(x,y)\\&-K_{21}(x)\partial_{x_1x_2}G_0(x,y)-K_{22}(x)\partial_{x_2x_2}G_0(x,y)\\&+F_1(x,y),
	\end{split}
	\end{equation*}
where $F_1(\cdot,y)\in L^q(U)$ for $1<q<2$ since $|\nabla\Gamma(x-y)|\leq \frac{1}{2\pi|x-y|} $.

 Denote $  z=\frac{T_x+T_y}{2}(x-y)  $ and $ T_x=\begin{pmatrix}
T_{11}(x) & T_{12}(x) \\
T_{21}(x) &  T_{22}(x)
\end{pmatrix}. $
One computes directly that
	\begin{equation}\label{a2}
	\begin{split}
-\nabla&_x	\cdot(K(x)\nabla_x G_0(x,y))\\
	=&-\frac{\sqrt{\det K(x)}^{-1}+\sqrt{\det K(y)}^{-1}}{2}\bigg(K_{11}(x)\left (\frac{T_{11}(x)+T_{11}(y)}{2}\partial_{z_1}+\frac{T_{21}(x)+T_{21}(y)}{2}\partial_{z_2}\right )^2\\
	&+K_{12}(x)\left (\frac{T_{11}(x)+T_{11}(y)}{2}\partial_{z_1}+\frac{T_{21}(x)+T_{21}(y)}{2}\partial_{z_2}\right )\left (\frac{T_{12}(x)+T_{12}(y)}{2}\partial_{z_1}+\frac{T_{22}(x)+T_{22}(y)}{2}\partial_{z_2}\right )\\
	&+K_{21}(x)\left (\frac{T_{11}(x)+T_{11}(y)}{2}\partial_{z_1}+\frac{T_{21}(x)+T_{21}(y)}{2}\partial_{z_2}\right )\left (\frac{T_{12}(x)+T_{12}(y)}{2}\partial_{z_1}+\frac{T_{22}(x)+T_{22}(y)}{2}\partial_{z_2}\right )\\
	&+K_{22}(x)\left (\frac{T_{12}(x)+T_{12}(y)}{2}\partial_{z_1}+\frac{T_{22}(x)+T_{22}(y)}{2}\partial_{z_2}\right )^2\bigg)\Gamma(z)+F_2(x,y)
	\end{split}
	\end{equation}
for some $F_2(\cdot,y)\in L^q(U)(1<q<2) $. Define
\begin{equation*}
\begin{split}
c_{11}(x,y)=&K_{11}(x)\left (\frac{T_{11}(x)+T_{11}(y)}{2}\right)^2+ 2K_{12}(x)\left (\frac{T_{11}(x)+T_{11}(y)}{2}\right)\left (\frac{T_{12}(x)+T_{12}(y)}{2}\right)\\
&+K_{22}(x)\left (\frac{T_{12}(x)+T_{12}(y)}{2}\right)^2,
\end{split}
\end{equation*}
\begin{equation*}
\begin{split}
c_{12}&(x,y)\\
=&2K_{11}(x)\left (\frac{T_{11}(x)+T_{11}(y)}{2}\right)\left (\frac{T_{21}(x)+T_{21}(y)}{2}\right)+ 2K_{12}(x)\left (\frac{T_{11}(x)+T_{11}(y)}{2}\right)\left (\frac{T_{22}(x)+T_{22}(y)}{2}\right)\\
&+ 2K_{12}(x)\left (\frac{T_{21}(x)+T_{21}(y)}{2}\right)\left (\frac{T_{12}(x)+T_{12}(y)}{2}\right)+2K_{22}(x)\left (\frac{T_{12}(x)+T_{12}(y)}{2}\right)\left (\frac{T_{22}(x)+T_{22}(y)}{2}\right),
\end{split}
\end{equation*}
\begin{equation*}
\begin{split}
c_{22}(x,y)=&K_{11}(x)\left (\frac{T_{21}(x)+T_{21}(y)}{2}\right)^2+ 2K_{12}(x)\left (\frac{T_{21}(x)+T_{21}(y)}{2}\right)\left (\frac{T_{22}(x)+T_{22}(y)}{2}\right)\\
&+K_{22}(x)\left (\frac{T_{22}(x)+T_{22}(y)}{2}\right)^2,
\end{split}
\end{equation*}
then
\begin{equation*}
\begin{split}
-&\nabla_x	\cdot(K(x)\nabla_x G_0(x,y))\\
&=-\frac{\sqrt{\det K(x)}^{-1}+\sqrt{\det K(y)}^{-1}}{2}(C_{11}(x,y)\partial_{z_1z_1}+C_{12}(x,y)\partial_{z_1z_2}+C_{22}(x,y)\partial_{z_2z_2})\Gamma(z)+F_2(x,y).
\end{split}
\end{equation*}
When $ x=y $, direct computation shows that
	\begin{equation}\label{a3}
	\begin{split}
	C_{12}(y,y)
	=&2(K_{11}T_{11}T_{21}+K_{12}T_{11}T_{22}+K_{12}T_{12}T_{21}+K_{22}T_{12}T_{22})(y)\\
	=&2((T^{-1}_{11}T^{-1}_{11}+T^{-1}_{12}T^{-1}_{12})T_{11}T_{21}+(T^{-1}_{11}T^{-1}_{21}+T^{-1}_{12}T^{-1}_{22})T_{11}T_{22}\\
	&+(T^{-1}_{11}T^{-1}_{21}+T^{-1}_{12}T^{-1}_{22})T_{12}T_{21}+(T^{-1}_{21}T^{-1}_{21}+T^{-1}_{22}T^{-1}_{22})T_{12}T_{22})(y)\\
	=&2((T_{11}T^{-1}_{11}+T_{12}T^{-1}_{21})T^{-1}_{11}T_{21}+(T_{21}T^{-1}_{12}+T_{22}T^{-1}_{22})T_{12}^{-1}T_{11}\\
	&+(T_{11}T_{11}^{-1}+T_{12}T_{21}^{-1})T_{21}^{-1}T_{22}+(T_{21}T_{12}^{-1}+T_{22}T_{22}^{-1})T_{22}^{-1}T_{12})(y)\\
	=&2((T^{-1}_{11}T_{21}+T_{21}^{-1}T_{22})+(T_{12}^{-1}T_{11}+T_{22}^{-1}T_{12}))(y)\\
	=&0.
	\end{split}
	\end{equation}
	Similarly,
	\begin{equation}\label{a4}
	C_{11}(y,y)=C_{22}(y,y)=1.
	\end{equation}
Thus by \eqref{a3}, \eqref{a4} and the regularity of $ T $, one has
	\begin{equation}\label{a5}
	C_{12}(x,y)=O(|x-y|),\ \ C_{ii}(x,y)-1=O(|x-y|),\ \ i=1,2.
	\end{equation}
Substituting \eqref{a5} into \eqref{a2}, we conclude that	
	\begin{equation}\label{2-14}
	\begin{split}
	-\nabla_x\cdot(K(x)\nabla_x G_0(x,y))
	=&-\frac{\sqrt{\det K(x)}^{-1}+\sqrt{\det K(y)}^{-1}}{2}\Delta_z\Gamma(z)+F(x,y),
	\end{split}
	\end{equation}
for some $F(\cdot,y)\in L^q(U)(1<q<2)$, which  implies that $ -\nabla_x\cdot(K(\cdot)\nabla_x G_0(\cdot,y))=F(\cdot,y) $ in any subdomain of $ U\setminus \{y\}. $

For fixed $y\in U$, let $S_K(\cdot,y)\in W^{1,2}(U)$ be the unique weak solution to the following Dirichlet problem
	\begin{align}\label{2-15}
	\left\{
	\begin{aligned}
	&-\nabla_x\cdot(K(x)\nabla_x S_K(x,y))=-F(x,y)\ \ \ \text{in}\ U,\\
	&S_K(x,y)=-G_0(x,y)\ \ \ \ \ \ \ \ \ \ \ \ \ \ \ \ \ \ \ \ \ \ \ \ \ \text{on}\ \partial U.\\
	\end{aligned}
	\right.
	\end{align}
Since $ K $ is smooth and positive definite, by classical elliptic regularity estimates (see \cite{GT}),   we have $ S_K(\cdot, y)\in W^{2,q}(U)  $ for every $ 1<q<2 $.
By Sobolev embedding theorem,  $ S_K(\cdot,y)\in C^{0,\gamma}(\overline{U}) $ for every $ \gamma\in(0,1) $.

Moreover, from the definition of $ F  $ and $ G_0 $, we have
\begin{equation*}
||F(\cdot, y)||_{L^q(U)}\leq \bar{C}_1\left (\int_{B_{diam(U)}(0)}\left (\ln\frac{1}{|z|}\right )^{q}+\left (\frac{1}{|z|}\right )^qdz\right )^{\frac{1}{q}}\leq \bar{C}_2\ \ \ \ \ \forall\ y\in U,
\end{equation*}
and
\begin{equation*}
-G_0(x,y)\leq \bar{C}_3\ln|x-y|\leq \bar{C}_4 \ \ \ \ \ \forall\ x,y \in U,
\end{equation*}
for some $ \bar{C}_i>0 (i=1,2,3,4)$. Thus by \eqref{2-15}, we have \eqref{2-12}.

Now  we prove \eqref{2-11}. We first prove that for any $u\in W^{2,\bar{p}}\cap W^{1,\bar{p}}_0(U) $ with $\bar{p}>2$,
	\begin{equation}\label{2-13}
\begin{split}
u(y)=\int_{U}G_K(x,y)\mathcal{L}_Ku(x)dx\ \ \ \ \forall\ y\in U.
\end{split}
\end{equation}
Fix any point $p\in U$. Let  $ \bar{\epsilon}>0 $ sufficiently small such that $ B_{\bar{\epsilon}}(p)\Subset U$. Then by \eqref{2-14} and \eqref{2-15}, we have $ G_K(\cdot,p)\in W^{2,q}(U\setminus B_{\bar{\epsilon}}(p))$ and
	\begin{align*}
	\left\{
	\begin{aligned}
	&-\nabla_x\cdot(K(x)\nabla_x G_K(x,p))=0\ \ \ \text{in}\ U\setminus B_{\bar{\epsilon}}(p),\\
	&G_K(x,p)=0\ \ \ \ \ \ \ \ \ \ \ \ \ \ \ \ \ \ \ \ \ \ \ \ \ \ \text{on}\ \partial U.\\
	\end{aligned}
	\right.
	\end{align*}
	Let
	\begin{align*}
	\left\{
	\begin{aligned}
	&x'=T_px,\ p'=T_pp,\\
	&\tilde{G}_K(x',p')=G_K(T^{-1}_px',T^{-1}_pp'),\\
	&U_p=\left\{x'=T_px\mid x\in U\right\}.
	\end{aligned}
	\right.
	\end{align*}
	Direct computation shows that
	\begin{align}\label{a6}\
	\left\{
	\begin{aligned}
	&\tilde{L}\tilde{G}_K=-\nabla_{x'}\cdot(\tilde{K}(x')\nabla_{x'} \tilde{G}_K(x',p'))=0\ \ \ \ \text{in}\ U_p\setminus B_{\bar{\epsilon}}(p'),\\
	&\tilde{G}_K(x',p')=0\ \ \ \ \ \ \ \ \ \ \ \ \ \ \ \ \ \ \ \ \ \ \ \ \ \ \ \ \ \ \ \ \ \ \ \ \ \ \text{on}\ \partial U_p,\\
	\end{aligned}
	\right.
	\end{align}
	where
	\begin{align*}
	&\tilde{K}_{11}(x')=K_{11}(x)T_{11}^2(p)+2K_{12}(x)T_{11}(p)T_{12}(p)+K_{22}(x)T_{12}^2(p),\\
	&\tilde{K}_{12}(x')=K_{11}(x)T_{11}(p)T_{21}(p)+K_{12}(x)T_{11}(p)T_{22}(p)+K_{12}(x)T_{12}(p)T_{21}(p)+K_{22}(x)T_{12}(p)T_{22}(p),\\
	&\tilde{K}_{21}(x')=K_{11}(x)T_{11}(p)T_{21}(p)+K_{12}(x)T_{11}(p)T_{22}(p)+K_{12}(x)T_{12}(p)T_{21}(p)+K_{22}(x)T_{12}(p)T_{22}(p),\\
	&\tilde{K}_{22}(x')=K_{11}(x)T_{21}^2(p)+2K_{12}(x)T_{21}(p)T_{22}(p)+K_{22}(x)T_{22}^2(p).
	\end{align*}

Let $v(x')=u(T^{-1}_px')$. We denote by $q\in (1,2)$ the conjugate exponent of $\bar{p}$. By Green's formula and \eqref{a6}, we have
	\begin{equation}\label{b0}
	\begin{split}
	\int_{U_p\setminus B_{\bar{\epsilon}}(p')}\tilde{G}_K(x',p')\tilde{L}v(x')dx'
	=&\int_{\partial B_{\bar{\epsilon}}(p')}v(x')\left (\tilde{K}(x')\nabla_{x'}\tilde{G}_K(x',p')|\frac{p'-x'}{{\bar{\epsilon}}}\right )dS_{x'}\\
	&-\int_{\partial B_{\bar{\epsilon}}(p')}\tilde{G}_K(x',p')\left (\tilde{K}(x')\nabla_{x'}v(x')|\frac{p'-x'}{{\bar{\epsilon}}}\right )dS_{x'}\\
	=:&I_1-I_2.
	\end{split}
	\end{equation}
Denote
\begin{equation*}
\tilde{G}_0(x',p')=G_0(T_p^{-1}x', T_p^{-1}p'), \ \ \ \ \tilde{S}_K(x',p')=S_K(T_p^{-1}x', T_p^{-1}p').
\end{equation*}	
Then $ \tilde{G}_K(\cdot,p')=\tilde{G}_0(\cdot,p')+\tilde{S}_K(\cdot,p') $ in $ U_p $.	
	
For $I_2$, from the definition of $\tilde{G}_K$ and the fact that $\nabla v\in W^{1,\bar{p}}(U_p)\subset C(\overline{U_p})$, $\tilde{S}_K(\cdot,p')\in W^{2,q}(U_p)\subset C(\overline{U_p})$, we get
	\begin{equation}\label{b1}
	\begin{split}
	I_2=&\int_{\partial B_{\bar{\epsilon}}(p')}\tilde{G}_0(x',p')\left (\tilde{K}(x')\nabla_{x'}v(x')|\frac{p'-x'}{{\bar{\epsilon}}}\right )dS_{x'}\\
	&+\int_{\partial B_{\bar{\epsilon}}(p')}\tilde{S}_K(x',p')\left (\tilde{K}(x')\nabla_{x'}v(x')|\frac{p'-x'}{{\bar{\epsilon}}}\right )dS_{x'}\\
	=&O({\bar{\epsilon}}|\ln{\bar{\epsilon}}|+{\bar{\epsilon}}).
	\end{split}
	\end{equation}
For $I_1$, notice that
	\begin{equation}\label{b2}
	\begin{split}
	I_1=&\int_{\partial B_{\bar{\epsilon}}(p')}v(x')\left (\tilde{K}(x')\nabla_{x'}\tilde{G}_0(x',p')|\frac{p'-x'}{{\bar{\epsilon}}}\right )dS_{x'}\\
	&\int_{\partial B_{\bar{\epsilon}}(p')}v(x')\left (\tilde{K}(x')\nabla_{x'}\tilde{S}_K(x',p')|\frac{p'-x'}{{\bar{\epsilon}}}\right )dS_{x'}\\
	=:&J_1+J_2.
	\end{split}
	\end{equation}
On the one hand, by the trace theorem
	\begin{equation}\label{b3}
	\begin{split}
	|J_2|\leq&\int_{\partial B_{\bar{\epsilon}}(p')}|v(x')||\tilde{K}(x')\nabla_{x'}\tilde{S}_K(x',p')|dS_{x'}\\
	\leq& C\int_{\partial B_{\bar{\epsilon}}(p')}|\partial_{x_1'}\tilde{S}_K(x',p')|+|\partial_{x_2'}\tilde{S}_K(x',p')|dS_{x'}\\
	=&C\int_{\partial B_1(0)}|\partial_{z_1}\tilde{S}_K({\bar{\epsilon}} z+p,p')|+|\partial_{z_2}\tilde{S}_K({\bar{\epsilon}} z+p,p')|dS_{z}\\
	\leq& C\bigg(\int_{B_1(0)}|\partial_{z_1}\tilde{S}_K({\bar{\epsilon}} z+p,p')|+|\partial_{z_2}\tilde{S}_K({\bar{\epsilon}} z+p,p')|dz\\
	&+\int_{B_1(0)}|\partial_{z_1z_1}\tilde{S}_K({\bar{\epsilon}} z+p,p')|+|\partial_{z_1z_2}\tilde{S}_K({\bar{\epsilon}} z+p,p')|+|\partial_{z_2z_2}\tilde{S}_K({\bar{\epsilon}} z+p,p')|dz\bigg)\\
	=&C\bigg(\frac{1}{{\bar{\epsilon}}}\int_{B_{\bar{\epsilon}}(p')}|\partial_{x_1'}\tilde{S}_K(x',p')|+|\partial_{x_2'}\tilde{S}_K(x',p')|dx'\\
	 &+\int_{B_{\bar{\epsilon}}(p')}|\partial_{x_1'x_1'}\tilde{S}_K(x',p')|+|\partial_{x_1'x_2'}\tilde{S}_K(x',p')|+|\partial_{x_2'x_2'}\tilde{S}_K(x',p')|dx'\bigg)\\
\to	& 0 \ \ \ \ \text{as}\ \ {\bar{\epsilon}}\to0,
	\end{split}
	\end{equation}
where we have used $ \tilde{S}_K(\cdot,p') \in W^{2,q}(U_p)$ for $ q\in(1,2). $
On the other hand,
	\begin{equation}\label{b4}
	\begin{split}
	J_1&=\int_{\partial B_{\bar{\epsilon}}(p')}v(x')\frac{\sqrt{\det K(T_p^{-1}x')}^{-1}+\sqrt{\det K(T_p^{-1}p')}^{-1}}{2}\left( \tilde{K}(x')\nabla_{x'}\Gamma(x'-p')|\frac{p'-x'}{{\bar{\epsilon}}}\right) dS_{x'}+o(1)\\
	&=\frac{1}{2\pi{\bar{\epsilon}}}\int_{\partial B_{{\bar{\epsilon}}}(p')}\tilde{K}_{11}(x')\frac{\sqrt{\det K(T_p^{-1}x')}^{-1}+\sqrt{\det K(T_p^{-1}p')}^{-1}}{2}v(x')dS_{x'}+o(1)\\
	&=\sqrt{\det K(p)}^{-1}v(p')+o(1)\\
	&=\sqrt{\det K(p)}^{-1}u(p)+o(1).
	\end{split}
	\end{equation}
	Applying \eqref{b1}, \eqref{b2}, \eqref{b3} and \eqref{b4} to \eqref{b0} and letting ${\bar{\epsilon}}\to 0$, we get
	\begin{equation*}
	\begin{split}
	\sqrt{\det K(p)}^{-1}u(p)=&\int_{U_p}\tilde{G}_K(x',p')\tilde{L}v(x')dx'\\
	=&\det (T_p)\int_{U}{G}_K(x,p)\mathcal{L}_Ku(x)dx\\
	=&\sqrt{\det K(p)}^{-1}\int_{U}{G}_K(x,p)\mathcal{L}_Ku(x)dx.
	\end{split}
	\end{equation*}
Then \eqref{2-13} holds.
	
Finally, we observe that if $ f_1,f_2\in L^q(U) $, then there holds
\begin{equation*}
\begin{split}
&\int_U f_1\mathcal{G}_K f_2-\iint_{U\times U}G_0(y,x)f_1(x)f_2(y)dxdy\\
&\ \ =\int_U f_2\mathcal{G}_K f_1-\iint_{U\times U}G_0(y,x)f_2(x)f_1(y)dxdy,
\end{split}
\end{equation*}	
and thus
\begin{equation*}
\iint_{U\times U}S_K(y,x)f_1(x)f_2(y)dxdy=\iint_{U\times U}S_K(y,x)f_2(x)f_1(y)dxdy.
\end{equation*}	
It follows that for every $ x,y\in U $, $ S_K(x, y) = S_K(y, x) $, and thus the function $ S_K\in C^{0,\gamma}_{loc}(U\times U). $  By \eqref{2-111}, \eqref{2-13} and Proposition \ref{exist of G_K}, we get \eqref{2-11}. The proof of Theorem \ref{expe of G_K} is therefore finished.

\end{proof}
\begin{remark}
One can not improve the regularity of  $ S_K$ to  $ C^{0,\gamma}(\overline{U\times U}) $. Since from \eqref{2-15}, for any $ x\in \partial U $, $ S_K(x,y)=-G_0(x,y)\to -\infty $ as $ y\to x. $ This implies that $ ||S_K||_{L^\infty(U\times U)} $ is unbounded.
\end{remark}

\section{Variational settings: Existence and profile of maximizers}
As an application, Theorem \ref{expe of G_K} can be used to construct concentrated helical vorticity field to 3D Euler equations in infinite pipes, which is shown in Theorem \ref{thm01}.
From subsection 1.2, it suffices to construct a family of concentrated solutions to \eqref{rot eq2} with $ f_\varepsilon(t)=\frac{1}{\varepsilon^2}\textbf{1}_{\{t>\mu^\varepsilon\}} $ for some $ \mu^\varepsilon $.
To this end, our first step is to introduce a certain variational problem and show the existence and profile of maximizers.

Let $k>0, d>0, R^*>0$  be three fixed numbers. For any $ r_*\in (0,R^*) $, from now on we will always choose  $ \alpha>0 $  such that
\begin{equation}\label{alpha}
\alpha=\frac{d}{4\pi k\sqrt{k^2+r_*^2}}.
\end{equation}
Denote $ \Omega=B_{R^*}(0) $ the cross-section of the infinite cylinder $ D $. Let $0<\varepsilon<\min\{1, \sqrt{|\Omega|/d}\}$ be a parameter. Define
\begin{equation*}
\mathcal{M}_{\varepsilon}:=\left \{\omega\in L^\infty(\Omega)~|~\int_{\Omega}\omega dx = d,\ \text{and} ~ 0\le \omega \le \frac{1}{\varepsilon^2} \right \}.
\end{equation*}
For any $ \omega\in \mathcal{M}_{\varepsilon} $, define the kinetic energy
\begin{equation}\label{KE}
\mathcal{E}(\omega):=\frac{1}{2}\int_{\Omega}\omega\mathcal{G}_{K_H}\omega dx,
\end{equation}
and the moment of inertia
\begin{equation}\label{MI}
\mathcal{I}(\omega):=\frac{1}{2}\int_{\Omega}|x|^2\omega  dx.
\end{equation}

Consider the maximization problem of the following functional over $\mathcal{M}_{\varepsilon}$
\begin{equation}\label{Energy}
\begin{split}
\mathcal{E}_\varepsilon(\omega):=&\mathcal{E}(\omega)-\alpha\ln\frac{1}{\varepsilon}\mathcal{I}(\omega)\\
=&\frac{1}{2}\int_{\Omega}\omega\mathcal{G}_{K_H}\omega dx-\frac{\alpha }{2}\ln\frac{1}{\varepsilon}\int_{\Omega}|x|^2\omega  dx.
\end{split}
\end{equation}
For any $\omega\in \mathcal{M}_{\varepsilon}$, by the classical elliptic estimate, we have $\mathcal{G}_{K_H}\omega\in W^{2,p}(\Omega)$ for any $1<p<+\infty$.  Thus $ \mathcal{E}_\varepsilon(\omega) $ is a well-defined functional on $ \mathcal{M}_{\varepsilon} $.

\begin{proposition}\label{exist of max}
There exists $\omega=\omega_{\varepsilon} \in \mathcal{M}_{\varepsilon} $ such that
	\begin{equation*}
	\mathcal{E}_\varepsilon(\omega_{\varepsilon})= \max_{\tilde{\omega} \in \mathcal{M}_{\varepsilon}}\mathcal{E}_\varepsilon(\tilde{\omega})<+\infty.
	\end{equation*}
	Moreover,
	\begin{equation}\label{3-1}
	\omega_{\varepsilon}=\frac{1}{\varepsilon^2}\mathbf{1}_{\{\psi^{\varepsilon}>0\}}  \ \ a.e.\  \text{in}\  \Omega,
	\end{equation}
	where
	\begin{equation}\label{3-2}
	\psi^{\varepsilon}=\mathcal{G}_{K_H}\omega_{\varepsilon}-\frac{\alpha |x|^2}{2}\ln\frac{1}{\varepsilon} -\mu^{\varepsilon},
	\end{equation}
	and the Lagrange multiplier $\mu^{\varepsilon} \ge -\frac{\alpha |R^*|^2}{2} \ln{\frac{1}{\varepsilon}}$ is determined by $\omega_{\varepsilon}$. Consequently, $ \omega_{\varepsilon} $ is a weak solution to \eqref{rot eq}  with $ f_{\varepsilon}(t)=\frac{1}{\varepsilon^2}\textbf{1}_{\{t>\mu^\varepsilon\}} $.
\end{proposition}

\begin{proof}
Clearly, it follows from the definition of $ \mathcal{M}_\varepsilon $ that
\begin{equation*}
\sup_{\tilde{\omega} \in \mathcal{M}_{\varepsilon}}\mathcal{E}_\varepsilon(\tilde{\omega})<+\infty.
\end{equation*}
We may take a sequence $\{\omega_{j}\}\subset \mathcal{M}_{\varepsilon}$ such that as $j\to +\infty$
	\begin{equation*}
	\begin{split}
	\mathcal{E}_\varepsilon(\omega_{j})  \to \sup\{\mathcal{E}_\varepsilon(\tilde{\omega})\mid\tilde{\omega}\in \mathcal{M}_{\varepsilon}\}.
	\end{split}
	\end{equation*}
Since $ \mathcal{M}_\varepsilon $ is closed in $ L^\infty(\Omega) $ weak star topology and $ L^2(\Omega) $ weak  topology, there exists $ \omega\in L^{\infty}(\Omega) $ such that
\begin{equation*}
\omega_{j}\to \omega \ \ \ \ \ \ ~~\text{weakly~in}\ L^{2}(\Omega)\ \text{and~weakly~star~in}\ L^{\infty}(\Omega).
\end{equation*}
It is easily checked that $\omega\in \mathcal{M}_{\varepsilon}$. Since $ \mathcal{G}_{K_H} $ is a bounded operator from $ L^2(\Omega) $ to $ W^{2,2}(\Omega) $, we first have
	\begin{equation*}
	\lim_{j\to +\infty}\int_\Omega{\omega_{j} \mathcal{G}_{K_H}\omega_{j}}dx = \int_\Omega{\omega \mathcal{G}_{K_H}\omega}dx.
	\end{equation*}
On the other hand, we have
	\begin{equation*}
	\begin{split}
	\lim_{j\to +\infty} \int_{\Omega}|x|^2\omega_j dx    = \int_{\Omega}|x|^2\omega  dx.
	\end{split}
	\end{equation*}
Consequently, we may conclude that $\mathcal{E}_\varepsilon(\omega)=\lim_{j\to +\infty}\mathcal{E}_\varepsilon(\omega_j)=\sup \mathcal{E}_\varepsilon$, with $\omega\in \mathcal{M}_{\varepsilon}$.
	
We now turn to prove $\eqref{3-1}$. Consider the family of variations of $\omega$
	\begin{equation*}
	\omega_{(s)}=\omega+s(\tilde{\omega}-\omega),\ \ \ s\in[0,1],
	\end{equation*}
defined for arbitrary $\tilde{\omega}\in \mathcal{M}_{\varepsilon}$. Since $\omega$ is a maximizer, we have $ \mathcal{E}_\varepsilon(\omega)\geq \mathcal{E}_\varepsilon(\omega_{(s)}) $, which implies that
	\begin{equation*}
	\begin{split}
	0 & \ge \frac{d}{ds}\mathcal{E}_\varepsilon(\omega_{(s)})|_{s=0^+} \\
	& =\int_{\Omega}(\tilde{\omega}-\omega)\left(\mathcal{G}_{K_H}\omega-\frac{\alpha |x|^2}{2} \ln{\frac{1}{\varepsilon}} \right)dx.
	\end{split}
	\end{equation*}
This shows that for any  $\tilde{\omega}\in \mathcal{M}_{\varepsilon}$, there holds
	\begin{equation*}
	\int_{\Omega}\omega \left(\mathcal{G}_{K_H}\omega-\frac{\alpha |x|^2}{2} \ln{\frac{1}{\varepsilon}}\right)dx \ge \int_{\Omega}\tilde{\omega}  \left(\mathcal{G}_{K_H}\omega-\frac{\alpha |x|^2}{2} \ln{\frac{1}{\varepsilon}} \right)dx.
	\end{equation*}
By an adaptation of the bathtub principle (see \cite{Lieb}, \S1.14), we obtain
	\begin{equation}\label{3-3}
	\begin{cases}
\mathcal{G}_{K_H}\omega-\frac{\alpha |x|^2}{2} \ln{\frac{1}{\varepsilon}}-\mu^{\varepsilon}\ge0, &\ \ \   \omega=\frac{1}{\varepsilon^2}, \\
\mathcal{G}_{K_H}\omega-\frac{\alpha |x|^2}{2} \ln{\frac{1}{\varepsilon}}-\mu^{\varepsilon}=0, &\ \ \   0<\omega<\frac{1}{\varepsilon^2}, \\
\mathcal{G}_{K_H}\omega-\frac{\alpha |x|^2}{2} \ln{\frac{1}{\varepsilon}}-\mu^{\varepsilon}\le 0, &\ \ \   \omega=0,
	\end{cases}
	\end{equation}
where $\mu^\varepsilon $ arises as a Lagrange multiplier satisfying
\begin{equation}\label{def of mu}
\mu^\varepsilon =\inf\left\{s\in\mathbb{R}\,\,\,\mid\,\,\, \bigg|\left\{x\in \Omega\,\,\mid\,\, \mathcal{G}_{K_H}\omega-\frac{\alpha |x|^2}{2} \ln{\frac{1}{\varepsilon}}>s\right\}\bigg|\leq d\varepsilon^2\right\}.
\end{equation}
This implies that $ \omega=\frac{1}{\varepsilon^2} $ a.e. in $ \{x\in\Omega\mid \mathcal{G}_{K_H}\omega-\frac{\alpha |x|^2}{2} \ln{\frac{1}{\varepsilon}}-\mu^{\varepsilon}>0\} $ and $ \omega=0 $ a.e. in $ \{x\in\Omega\mid \mathcal{G}_{K_H}\omega-\frac{\alpha |x|^2}{2} \ln{\frac{1}{\varepsilon}}-\mu^{\varepsilon}<0\} $.

On the set $ \{x\in\Omega\mid \mathcal{G}_{K_H}\omega-\frac{\alpha |x|^2}{2} \ln{\frac{1}{\varepsilon}}-\mu^{\varepsilon}=0\} $, by properties of Sobolev space, we have
\begin{equation*}
\omega=\mathcal{L}_{K_H}(\mathcal{G}_{K_H}\omega)=\mathcal{L}_{K_H}\left (\frac{\alpha |x|^2}{2} \ln{\frac{1}{\varepsilon}}\right )=-\frac{2 k^4\alpha}{(k^2+|x|^2)^2}\ln\frac{1}{\varepsilon}<0\ \ \text{a.e.}
\end{equation*}
However, by the definition of $ \mathcal{M}_{\varepsilon} $, we have $ w\ge 0 $ a.e. in $ \Omega $. So the Lebesgue measure of the set $ \{x\in\Omega\mid \mathcal{G}_{K_H}\omega-\frac{\alpha |x|^2}{2} \ln{\frac{1}{\varepsilon}}-\mu^{\varepsilon}=0\} $ is zero. Let $ \omega_\varepsilon=\omega $.
Now the stated form $\eqref{3-1}$ follows immediately.

We prove the lower bound of $ \mu^\varepsilon $.  We shall argue by contradiction.  Suppose that $ \mu^\varepsilon<-\frac{\alpha |R^*|^2}{2} \ln{\frac{1}{\varepsilon}} $, then by \eqref{def of mu}, we get
\begin{equation}\label{3-4}
\bigg|\left \{x\in \Omega\mid \mathcal{G}_{K_H}\omega-\frac{\alpha |x|^2}{2} \ln{\frac{1}{\varepsilon}}>-\frac{\alpha |R^*|^2}{2} \ln{\frac{1}{\varepsilon}}\right \}\bigg|\leq d\varepsilon^2.
\end{equation}
Since $ \mathcal{G}_{K_H} $ is a positive  elliptic operator, we have $ \mathcal{G}_{K_H}\omega> 0 $ in $ \Omega $. Note also that  for any $ x\in \Omega $,  $ |x|\leq R^* $. Thus  we get $ \mathcal{G}_{K_H}\omega-\frac{\alpha |x|^2}{2} \ln{\frac{1}{\varepsilon}}>-\frac{\alpha |R^*|^2}{2} \ln{\frac{1}{\varepsilon}} $ in $ \Omega $. This implies that $ \bigg|\left \{x\in \Omega\mid \mathcal{G}_{K_H}\omega-\frac{\alpha |x|^2}{2} \ln{\frac{1}{\varepsilon}}>-\frac{\alpha |R^*|^2}{2} \ln{\frac{1}{\varepsilon}}\right \}\bigg|=|\Omega| $, which is clearly a contradiction to \eqref{3-4} since $ \varepsilon<\sqrt{|\Omega|/d} $. By \eqref{3-1} , it is not hard to check that $ \omega_{\varepsilon} $ is a weak solution to \eqref{rot eq}. The proof is thus completed.
\end{proof}

\section{Asymptotic behavior of maximizers $ \omega_\varepsilon $}
Now we are to obtain the limiting behavior of $  \omega_\varepsilon $ as $ \varepsilon $ tends to 0. The key is to use the structure of Green's function of the operator $ \mathcal{L}_{K_H} $ shown in Theorem \ref{expe of G_K}. For convenience we will use $ C $
to denote generic positive constants independent of $ \varepsilon $ that may change from line to line.

Let us define a function $ Y(x):\Omega\to\mathbb{R} $ which will be frequently used in this section. For any $ x\in \Omega $, define
\begin{equation}\label{def of Y}
Y(x):=\frac{d}{2\pi\sqrt{\det K_H(x)}}-\alpha|x|^2=\frac{d\sqrt{k^2+|x|^2}}{2\pi k}-\alpha|x|^2,
\end{equation}
where $ \alpha $ is chosen by \eqref{alpha}. Clearly, $ Y $ is radially symmetric in $ \Omega=B_{R^*}(0) $. Then one computes directly that
\begin{lemma}\label{Y max}
Under the choice of $ \alpha $ in \eqref{alpha},  the maximizers set of $ Y $ in $ \Omega $ is $ \{x\mid |x|=r_*\} $. That is, $ Y|_{\partial B_{r_*}(0)}=\max_{\Omega}Y $. Moreover, up to a rotation the maximizer is unique.
\end{lemma}

Let $G_{K_H}(x,y)$ be the Green's function of $\mathcal{L}_{K_H}$ in $\Omega$, with respect to zero Dirichlet data and the Lebesgue measure $dx$. That is, the solution of the  second-order elliptic equation in divergence form
\begin{equation*}
\begin{cases}
-\text{div}(K_H(x)\nabla u)=f,\  &\ \Omega,\\
u=0,\ &\ \partial \Omega
\end{cases}
\end{equation*}
can be expressed by
\begin{equation*}
u(\cdot)=\mathcal{G}_{K_H}f(\cdot)=\int_{\Omega}G_{K_H}(\cdot,y)f(y)dy.
\end{equation*}
By Theorem \ref{expe of G_K}, we have the following properties of $G_{K_H}$.

\begin{proposition}\label{expe of G_H}
Let $ p>2. $ There exists a function $ S_{K_H}\in C^{0,\gamma}_{loc} (\Omega\times \Omega) $ for some $ \gamma\in(0,1) $ such that for every $\omega\in L^p(\Omega)$ and every $ x\in \Omega $,
\begin{equation*}
\begin{split}
\mathcal{G}_{K_H}\omega(x)=&\int_{\Omega}G_{K_H}(x,y)\omega(y)dy\\
=&\int_{\Omega}\left (\frac{\sqrt{\det K_H(x)}^{-1}+\sqrt{\det K_H(y)}^{-1}}{2}\Gamma\left (\frac{T_x+T_y}{2}(x-y)\right )+S_{K_H}(x,y)\right)\omega(y)dy.
\end{split}
\end{equation*}
Here $ \Gamma(x)=-\frac{1}{2\pi}\ln|x|$ is the fundamental solution of the Laplacian $ -\Delta $ in $ \mathbb{R}^2 $ and $ T_x $ is a $ C^\infty $ positive-definite matrix-valued function determined by $ K_H $ satisfying
\begin{equation}\label{matrix T}
T_x^{-1}(T_x^{-1})^{t}=K_H(x)\ \ \ \ \forall\ x\in \Omega.
\end{equation}
Moreover, $ S_{K_H}(x,y)=S_{K_H}(y,x) $ for any $ x,y\in\Omega $ and
\begin{equation*}
S_{K_H}(x,y)\leq C \ \ \ \ \forall\ x,y \in \Omega.
\end{equation*}
\end{proposition}

\subsection{Lower bound of $ \mathcal{E}_\varepsilon(\omega_\varepsilon) $}
We begin by giving a lower bound of $ \mathcal{E}_\varepsilon(\omega_\varepsilon) $.

\begin{lemma}\label{lowbdd of E}
For any $z\in \Omega$, there exists $C>0$ such that for all $\varepsilon$ sufficiently small,
	\begin{equation*}
\mathcal{E}_\varepsilon(\omega_\varepsilon)\ge \frac{d}{2}Y(z)\ln\frac{1}{\varepsilon}-C.
	\end{equation*}
As a consequence, there holds
\begin{equation}\label{4-1}
\mathcal{E}_\varepsilon(\omega_\varepsilon)\ge \frac{d}{2}Y((r_*,0))\ln\frac{1}{\varepsilon}-C.
\end{equation}
\end{lemma}

\begin{proof}
For any $z\in \Omega$, we choose a test function $\tilde{\omega}_{\varepsilon} $ defined by
\begin{equation*}
\tilde{\omega}_{\varepsilon}=\frac{1}{\varepsilon^2}\textbf{1}_{T_z^{-1}B_{r_\varepsilon}(0)+z},
\end{equation*}
where the matrix $ T_z $ satisfies \eqref{matrix T}, and $ r_\varepsilon>0 $ satisfies
\begin{equation}\label{4-2}
\pi r_\varepsilon^2\sqrt{\det K_H(z)}=\varepsilon^2d.
\end{equation}
Thus $ \tilde{\omega}_{\varepsilon} \in \mathcal{M}_{\varepsilon} $ for $ \varepsilon $ sufficiently small. Since $\omega_\varepsilon$ is a maximizer, we have $\mathcal{E}_\varepsilon(\omega_\varepsilon)\ge \mathcal{E}_\varepsilon(\tilde{\omega}_{\varepsilon})$. By Proposition \ref{expe of G_H}, we get
	\begin{equation*}
\begin{split}
\mathcal{E}_\varepsilon(\tilde{\omega}_{\varepsilon})=& \frac{1}{2}\int_\Omega\int_\Omega\tilde{\omega}_{\varepsilon}(x)G_{K_H}(x,y)\tilde{\omega}_{\varepsilon}(y)dxdy-{\frac{\alpha}{2} \ln{\frac{1}{\varepsilon}}}\int_{\Omega}|x|^2\tilde{\omega}_{\varepsilon}dx\\
\geq&\frac{1}{2}\int_{T_z^{-1}B_{r_\varepsilon}(0)+z}\int_{T_z^{-1}B_{r_\varepsilon}(0)+z}\frac{\sqrt{\det K_H(x)}^{-1}+\sqrt{\det K_H(y)}^{-1}}{2\varepsilon^4}\Gamma\left (\frac{T_x+T_y}{2}(x-y)\right) dxdy\\
& -{\frac{\alpha}{2} \ln{\frac{1}{\varepsilon}}}\int_{\Omega}|x|^2\tilde{\omega}_{\varepsilon}dx-C_1.
\end{split}
\end{equation*}
Since $ r_\varepsilon=O(\varepsilon),  $ by the positive-definiteness and regularity of $ K_H $, we have
\begin{equation*}
\sqrt{\det K_H(x)}^{-1}=\sqrt{\det K_H(z)}^{-1}+O(\varepsilon),\ \ \ \ \forall x\in T_z^{-1}B_{r_\varepsilon}(0)+z,
\end{equation*}
\begin{equation*}
\Gamma\left (\frac{T_x+T_y}{2}(x-y)\right)=\Gamma\left (T_z(x-y)\right)+O(\varepsilon),\ \ \ \ \forall x,y\in T_z^{-1}B_{r_\varepsilon}(0)+z,\ x\neq y.
\end{equation*}
Thus one computes directly that
	\begin{equation*}
	\begin{split}
\mathcal{E}_\varepsilon(\tilde{\omega}_{\varepsilon})
\geq&\frac{1}{2\varepsilon^4}\int_{T_z^{-1}B_{r_\varepsilon}(0)+z}\int_{T_z^{-1}B_{r_\varepsilon}(0)+z}\left (\sqrt{\det K_H(z)}^{-1}+O(\varepsilon)\right ) \left( \Gamma\left (T_z(x-y)\right)+O(\varepsilon)\right ) dxdy\\
& -{\frac{d\alpha(|z|+O(\varepsilon))^2}{2} \ln{\frac{1}{\varepsilon}}}-C_2\\
\ge& \frac{\sqrt{\det K_H(z)}^{-1}}{4\pi\varepsilon^4}\int_{T_z^{-1}B_{r_\varepsilon}(0)+z}\int_{T_z^{-1}B_{r_\varepsilon}(0)+z}\ln{\frac{1}{|T_z(x-y)|}}dxdy -{\frac{d\alpha|z|^2}{2} \ln{\frac{1}{\varepsilon}}}-C_3\\
=& \frac{\sqrt{\det K_H(z)}^{-1}}{4\pi\varepsilon^4}\int_{B_{r_\varepsilon}(0)}\int_{B_{r_\varepsilon}(0)}\ln{\frac{1}{|x'-y'|}}\cdot \det K_H(z)dx'dy' -{\frac{d\alpha|z|^2}{2} \ln{\frac{1}{\varepsilon}}}-C_3\\
\geq&\frac{\sqrt{\det K_H(z)}}{4\pi\varepsilon^4}(\pi r_\varepsilon^2)^2\ln\frac{1}{\varepsilon}-{\frac{d\alpha|z|^2}{2} \ln{\frac{1}{\varepsilon}}}-C_4\\
=& \frac{d}{2}Y(z)\ln\frac{1}{\varepsilon}-C_4,
	\end{split}
	\end{equation*}
where we have used \eqref{matrix T} and \eqref{4-2}. By choosing $ z=(r_*,0) $ and using Lemma \ref{Y max}, we get \eqref{4-1}. The proof is thus finished.
\end{proof}

We now turn to estimate the Lagrange multiplier $\mu^{\varepsilon}$.
\begin{lemma}\label{lowbdd of mu}
There holds for $ \varepsilon $ sufficiently small,
	\[\mu^{\varepsilon}\ge \frac{2}{d}\mathcal{E}_\varepsilon(\omega_\varepsilon)+ \frac{\alpha}{2d}\ln\frac{1}{\varepsilon}\int_\Omega |x|^2 \omega_\varepsilon dx-C.\]
As a consequence,
\begin{equation*}
\mu^{\varepsilon}\ge Y((r_*,0))\ln\frac{1}{\varepsilon}+ \frac{\alpha}{2d}\ln\frac{1}{\varepsilon}\int_\Omega |x|^2 \omega_\varepsilon dx-C.
\end{equation*}
\end{lemma}

\begin{proof}
It follows from the definition of $ \mathcal{E}_\varepsilon $ and \eqref{3-2} that
\begin{equation}\label{4-5}
\begin{split}
2\mathcal{E}_\varepsilon(\omega_\varepsilon)=&\int_{\Omega}\omega_\varepsilon\mathcal{G}_{K_H}\omega_\varepsilon dx-\alpha \ln\frac{1}{\varepsilon}\int_{\Omega}|x|^2\omega_\varepsilon  dx\\
=&\int_{\Omega}\omega_\varepsilon\psi^\varepsilon dx-\frac{\alpha}{2} \ln\frac{1}{\varepsilon}\int_{\Omega}|x|^2\omega_\varepsilon  dx+d\mu^\varepsilon.
\end{split}
\end{equation}	

By \eqref{3-2} and $\mu^{\varepsilon} \ge -\frac{\alpha |R^*|^2}{2} \ln{\frac{1}{\varepsilon}}$, we have $ \psi^\varepsilon\leq0 $ on $ \partial \Omega. $ So $ (\psi^\varepsilon)_+\in H^1_0(\Omega). $ Denote $ \bar{A}_\varepsilon=\{x\in \Omega\mid \psi^\varepsilon>0\} $ the vortex core of $ \Omega_\varepsilon. $ By Proposition \ref{exist of max}, we know that $ supp(\omega_\varepsilon)=\bar{A}_\varepsilon. $
	
Since $ \mathcal{L}_{K_H} \psi^\varepsilon=\omega_\varepsilon-\frac{\alpha}{2}\mathcal{L}_{K_H}|x|^2\ln\frac{1}{\varepsilon}$, multiplying  $ (\psi^\varepsilon)_+ $ to both sides of this equation and integrating by parts, we get
\begin{equation}\label{4-3}
\int_\Omega(K_H(x)\nabla \psi^\varepsilon| \nabla(\psi^\varepsilon)_+)dx=\int_\Omega \omega_\varepsilon(\psi^\varepsilon)_+dx-\frac{\alpha}{2}\ln\frac{1}{\varepsilon}\int_\Omega\mathcal{L}_{K_H}|x|^2(\psi^\varepsilon)_+dx.
\end{equation}

On the other hand, since $ \mathcal{L}_{K_H}|x|^2=-\frac{4 k^4}{(k^2+|x|^2)^2}\in L^\infty(\Omega) $, we have
\begin{equation}\label{4-4}
\begin{split}
\int_\Omega \omega_\varepsilon&(\psi^\varepsilon)_+dx-\frac{\alpha}{2}\ln\frac{1}{\varepsilon}\int_\Omega\mathcal{L}_{K_H}|x|^2(\psi^\varepsilon)_+dx\\
\leq &\frac{C}{\varepsilon^2}\int_{\bar{A}_\varepsilon} (\psi^\varepsilon)_+dx\\
\leq& \frac{C|\bar{A}_\varepsilon|^{\frac{1}{2}}}{\varepsilon^2}\left (\int_{\Omega} (\psi^\varepsilon)_+^2dx\right )^{\frac{1}{2}}\\
\leq&\frac{C|\bar{A}_\varepsilon|^{\frac{1}{2}}}{\varepsilon^2}\int_{\bar{A}_\varepsilon} |\nabla (\psi^\varepsilon)_+|dx\\
\leq&\frac{C|\bar{A}_\varepsilon|}{\varepsilon^2}\left (\int_{\Omega} |\nabla (\psi^\varepsilon)_+|^2dx\right )^{\frac{1}{2}}\\
\leq&Cd\left( \int_\Omega(K_H(x)\nabla \psi^\varepsilon| \nabla(\psi^\varepsilon)_+)dx\right)^{\frac{1}{2}},
\end{split}
\end{equation}	
where we have used the Sobolev imbedding $ W^{1,1}_0(\Omega)\subset L^2(\Omega) $, the positive-definiteness of $ K_H $ and the fact that $ \frac{|\bar{A}_\varepsilon|}{\varepsilon^2}=\int_\Omega\omega_\varepsilon dx=d. $

Combining \eqref{4-3} and \eqref{4-4}, we conclude that $\int_\Omega \omega_\varepsilon(\psi^\varepsilon)_+dx $ is uniformly bounded with respect to $\varepsilon$. The desired result clearly follows from \eqref{4-5} and Lemma \ref{lowbdd of E}.
\end{proof}

\subsection{Asymptotic behavior of support set of $ \omega_\varepsilon $}

Now we analyze the limiting behavior of the maximizer $ \omega_\varepsilon $.  We will prove that, to maximize the energy $ \mathcal{E}_\varepsilon $,  the support set of $ \omega_\varepsilon $ must shrink to a single point which is a maximizer  of $ Y $ in $ \Omega $ as $ \varepsilon $ tends to 0. Thus from Lemma \ref{Y max}, $ supp(\omega_\varepsilon) $ will concentrate near $ \partial B_{r_*}(0) $.

Let
\begin{equation}\label{111A}
\bar{P}_\varepsilon=\inf\{|x|\mid x\in {supp}(\omega_\varepsilon)\},\ \ \text{and}\ \ \bar{Q}_\varepsilon=\sup\{|x|\mid x\in {supp}(\omega_\varepsilon)\}.
\end{equation}
$ \bar{P}_\varepsilon $ and  $ \bar{Q}_\varepsilon $ describe the lower bound and upper bound of the distance between the origin and $ {supp}(\omega_{{\varepsilon}}) $, respectively. Let $ P_\varepsilon, Q_\varepsilon $ be two points in $ {supp}(\omega_\varepsilon) $ such that
\begin{equation*}
|P_\varepsilon|=\bar{P}_\varepsilon,\ \ \  \text{and}\ \ \  |Q_\varepsilon|=\bar{Q}_\varepsilon.
\end{equation*}

We now prove that  the support set of the maximizer $ \omega_\varepsilon $ constructed above
must be concentrated as $\varepsilon$ tends to zero. We reach our goal by several steps as follows.
\begin{lemma}\label{le6-1}
	$\lim\limits_{\varepsilon\to 0^+}\bar{P}_\varepsilon=r_*$.
\end{lemma}

\begin{proof}
Let  $\gamma\in(0,1)$. By Proposition \ref{exist of max}, for any  $x_\varepsilon\in {supp}(\omega_\varepsilon)$, we have $$\mathcal{G}_{K_H}\omega_\varepsilon(x_\varepsilon)-\frac{\alpha|x_\varepsilon|^2}{2}\ln{\frac{1}{\varepsilon}}\ge \mu^{{\varepsilon}}.$$
Note that
	\begin{equation*}
	\begin{split}
	\mathcal{G}_{K_H}\omega_\varepsilon(x_\varepsilon)&= \int_{\Omega}G_{K_H}(x_\varepsilon,y)\omega_\varepsilon(y)dy\\
	&=\left (\int_{\Omega\cap\{|T_{x_\varepsilon}(y-x_\varepsilon)|>\varepsilon^\gamma\}}+\int_{\Omega\cap\{|T_{x_\varepsilon}(y-x_\varepsilon)|\le\varepsilon^\gamma\}}\right )G_{K_H}(x_\varepsilon,y)\omega_\varepsilon(y)dy.
	\end{split}
	\end{equation*}

On the one hand, it follows from Proposition \ref{expe of G_H} that $  \max_{x,y\in\Omega}S_{K_H}(x,y)\le C $, which implies that $ \int_\Omega S_{K_H}(x_\varepsilon,y)\omega_\varepsilon(y)dy\leq C $.  Direct computation shows that
\begin{equation*}
\sqrt{\det K_H(y)}^{-1}\leq \sqrt{\det K_H((R^*,0))}^{-1} \ \ \ \ \ \ \ \ \forall\ y\in\Omega,
\end{equation*}
\begin{equation*}
\bigg|\frac{T_{x_\varepsilon}+T_y}{2}(y-x_\varepsilon)\bigg|\geq C\varepsilon^\gamma\ \ \  \ \ \ \ \ \forall\ |T_{x_\varepsilon}(y-x_\varepsilon)|>\varepsilon^\gamma.
\end{equation*}
Thus by the decomposition  of $ G_{K_H} $ in Proposition \ref{expe of G_H}, we have
	\begin{equation}\label{212}
	\begin{split}
&\int_{\Omega\cap\{|T_{x_\varepsilon}(y-x_\varepsilon)|>\varepsilon^\gamma\}}G_{K_H}(x_\varepsilon,y)\omega_\varepsilon(y)dy\\
&\ \ \leq\int\limits_{\Omega\cap\{|T_{x_\varepsilon}(y-x_\varepsilon)|>\varepsilon^\gamma\}}\frac{\sqrt{\det K_H(x_\varepsilon)}^{-1}+\sqrt{\det K_H(y)}^{-1}}{2}\Gamma\left (\frac{T_{x_\varepsilon}+T_y}{2}\left (x_\varepsilon-y\right )\right )\omega_\varepsilon(y)dy+C\\
&\ \ \le \frac{\sqrt{\det K_H((R^*,0))}^{-1}}{2\pi} \int\limits_{\Omega\cap\{|T_{x_\varepsilon}(y-x_\varepsilon)|>\varepsilon^\gamma\}}\ln\frac{1}{\varepsilon^\gamma}\omega_\varepsilon(y)dy+C\\
&\ \ \le \frac{\sqrt{\det K_H((R^*,0))}^{-1}d\gamma}{2\pi}\ln\frac{1}{\varepsilon}+C.
	\end{split}
	\end{equation}
On the other hand, by the smoothness and positive-definiteness of $ K_H $, it is not hard to get that
\begin{equation*}
\sqrt{\det K_H(y)}^{-1}\leq \sqrt{\det K_H(x_\varepsilon)}^{-1}+C\varepsilon^\gamma,\ \ \ \ \forall\ y\in\Omega\cap\{|T_{x_\varepsilon}(y-x_\varepsilon)|\leq\varepsilon^\gamma\}.
\end{equation*}
Hence we get
\begin{equation*}
\begin{split}
&\int_{\Omega\cap\{|T_{x_\varepsilon}(y-x_\varepsilon)|\leq\varepsilon^\gamma\}}G_{K_H}(x_\varepsilon,y)\omega_\varepsilon(y)dy\\
&\ \ \leq \left (\sqrt{\det K_H(x_\varepsilon)}^{-1}+C\varepsilon^\gamma\right )\int\limits_{\Omega\cap\{|T_{x_\varepsilon}(y-x_\varepsilon)|\leq\varepsilon^\gamma\}}\Gamma\left (\frac{T_{x_\varepsilon}+T_y}{2}\left (x_\varepsilon-y\right )\right )\omega_\varepsilon(y)dy+C\\
&\ \ \leq \left (\sqrt{\det K_H(x_\varepsilon)}^{-1}+C\varepsilon^\gamma\right )\int\limits_{\Omega\cap\{|T_{x_\varepsilon}(y-x_\varepsilon)|\leq\varepsilon^\gamma\}}\Gamma\left (T_{x_\varepsilon}\left (x_\varepsilon-y\right )\right )\omega_\varepsilon(y)dy+C\\
&=\ \ \left (\sqrt{\det K_H(x_\varepsilon)}^{-1}+C\varepsilon^\gamma\right )\int_{|z|\leq\varepsilon^\gamma}\Gamma(z)\omega_\varepsilon(T_{x_\varepsilon}^{-1}z+x_\varepsilon)\sqrt{\det K_H(x_\varepsilon)}dz+C.
\end{split}
\end{equation*}

Denote $ \omega_\varepsilon^*(z)=\frac{1}{\varepsilon^2}\textbf{1}_{B_{t_\varepsilon}(0)} $ the non-negative decreasing rearrangement function of $ \omega_\varepsilon(T_{x_\varepsilon}^{-1}z+x_\varepsilon)\textbf{1}_{\{|z|\leq \varepsilon^\gamma\}} $. Here we choose $ t_\varepsilon>0 $ such that $ \int\limits_{|z|\leq\varepsilon^\gamma} \omega_\varepsilon(T_{x_\varepsilon}^{-1}z+x_\varepsilon) dz=\int\limits_{\mathbb{R}^2} \omega_\varepsilon^* dz $. By the rearrangement inequality, we have
\begin{equation*}
\begin{split}
\int_{|z|\leq\varepsilon^\gamma}\Gamma(z)\omega_\varepsilon(T_{x_\varepsilon}^{-1}z+x_\varepsilon)dz\le&\int_{\mathbb{R}^2} \Gamma(z)\omega_\varepsilon^*(z) dz\\
\leq&\frac{1}{2\pi}\ln\frac{1}{\varepsilon}\int_{|z|\leq\varepsilon^\gamma}\omega_\varepsilon(T_{x_\varepsilon}^{-1}z+x_\varepsilon) dz+C,
\end{split}
\end{equation*}
from which we deduce that
\begin{equation}\label{213}
	\begin{split}
&\int_{\Omega\cap\{|T_{x_\varepsilon}(y-x_\varepsilon)|\leq\varepsilon^\gamma\}}G_{K_H}(x_\varepsilon,y)\omega_\varepsilon(y)dy\\
&\ \ \le\left (\sqrt{\det K_H(x_\varepsilon)}^{-1}+C\varepsilon^\gamma\right )\frac{1}{2\pi}\ln\frac{1}{\varepsilon}\int\limits_{|z|\leq\varepsilon^\gamma}\omega_\varepsilon(T_{x_\varepsilon}^{-1}z+x_\varepsilon)\sqrt{\det K_H(x_\varepsilon)}dz+C\\
&\ \ \le	 \left (\sqrt{\det K_H(x_\varepsilon)}^{-1}+C\varepsilon^\gamma\right )\frac{1}{2\pi}\ln\frac{1}{\varepsilon}\int\limits_{\Omega\cap\{|T_{x_\varepsilon}(y-x_\varepsilon)|\leq\varepsilon^\gamma\}}\omega_\varepsilon dy+C\\
&\ \ \le	  \frac{\sqrt{\det K_H(x_\varepsilon)}^{-1}}{2\pi}\ln\frac{1}{\varepsilon}\int\limits_{\Omega\cap\{|T_{x_\varepsilon}(y-x_\varepsilon)|\leq\varepsilon^\gamma\}}\omega_\varepsilon dy+C.
	\end{split}
	\end{equation}
Therefore by Lemma \ref{lowbdd of mu}, \eqref{212} and \eqref{213}, we conclude that for any $x_\varepsilon\in supp(\omega_\varepsilon)$ and $ \gamma\in(0,1) $, there holds
	\begin{equation}\label{210}
	\begin{split}
&\frac{\sqrt{\det K_H(x_\varepsilon)}^{-1}}{2\pi}\ln\frac{1}{\varepsilon} \int\limits_{\Omega\cap\{|T_{x_\varepsilon}(y-x_\varepsilon)|\leq\varepsilon^\gamma\}}\omega_\varepsilon dy+\frac{\sqrt{\det K_H((R^*,0))}^{-1}d\gamma}{2\pi}\ln\frac{1}{\varepsilon}-\frac{\alpha|x_\varepsilon|^2}{2}\ln{\frac{1}{\varepsilon}} \\
&\ \ \ge  Y((r_*,0))\ln\frac{1}{\varepsilon}+ \frac{\alpha}{2d}\ln\frac{1}{\varepsilon}\int_\Omega |x|^2 \omega_\varepsilon dx-C.
	\end{split}
	\end{equation}
Dividing both sides of the above inequality by $\ln\frac{1}{\varepsilon}$, we obtain
	\begin{equation}\label{214}
	\begin{split}
Y(x_\varepsilon)\ge& \frac{\sqrt{\det K_H(x_\varepsilon)}^{-1}}{2\pi}\int\limits_{\Omega\cap\{|T_{x_\varepsilon}(y-x_\varepsilon)|\leq\varepsilon^\gamma\}}\omega_\varepsilon dy-\alpha|x_\varepsilon|^2\\
\ge &  Y((r_*,0))+\frac{\alpha}{2d}\left \{\int_\Omega |x|^2 \omega_\varepsilon dx-d|x_\varepsilon|^2\right \}-\frac{\sqrt{\det K_H((R^*,0))}^{-1}d\gamma}{2\pi}-\frac{C}{\ln\frac{1}{\varepsilon}}.
	\end{split}
	\end{equation}
	Note that
	\begin{equation*}
\int_\Omega |x|^2 \omega_\varepsilon dx \ge d (\bar{P}_\varepsilon) ^2.
	\end{equation*}
	Taking $ x_\varepsilon =P_\varepsilon$ and letting $\varepsilon$ tend to $0^+$, we deduce from $\eqref{214}$ that
	\begin{equation}\label{215}
	\liminf_{\varepsilon\to 0^+}Y(P_\varepsilon)\ge Y((r_*,0))-\sqrt{\det K_H((R^*,0))}^{-1}d\gamma/(2\pi).
	\end{equation}
	Hence we get the desired result by letting $\gamma \to 0$ and using Lemma \ref{Y max}.
\end{proof}

Next, we estimate the moment of inertia $ \mathcal{I}(\omega_\varepsilon) $.
\begin{lemma} \label{le6-2}
	\begin{equation}\label{111}
	\lim_{\varepsilon\to 0^+}\mathcal{I}(\omega_\varepsilon)=\frac{1}{2}\lim_{\varepsilon\to 0^+}\int_\Omega|x|^2\omega_\varepsilon dx= \frac{d}{2} r_*^2.
	\end{equation}
	As a consequence, for any $\eta>0$, there holds
	\begin{equation}\label{112}
	\lim_{\varepsilon\to 0^+}\int_{\Omega\cap\{x\mid |x|\ge r_*+\eta\}}\omega_\varepsilon dx=0.
	\end{equation}
\end{lemma}

\begin{proof}
Taking $ x_\varepsilon =P_\varepsilon$ in $\eqref{214}$, we know that for any $\gamma\in(0,1)$,
	\begin{equation*}
	0\le  \liminf_{\varepsilon\to 0^+}\left [\int_\Omega |x|^2 \omega_\varepsilon dx - d (\bar{P}_\varepsilon) ^2\right ]\le  \limsup_{\varepsilon\to 0^+}\left [\int_\Omega |x|^2 \omega_\varepsilon dx - d (\bar{P}_\varepsilon) ^2\right ]\le \frac{\sqrt{\det K_H((R^*,0))}^{-1}d^2\gamma}{\alpha\pi}.
	\end{equation*}
	Combining this with Lemma \ref{le6-1} we get
	\begin{equation*}
	\lim_{\varepsilon\to 0^+}\int_\Omega |x|^2 \omega_\varepsilon dx   =\lim_{\varepsilon\to 0^+}d (\bar{P}_\varepsilon) ^2  =d r_*^2.
	\end{equation*}
	\eqref{112} follows immediately from $\eqref{111}$ and Lemma \ref{le6-1}.
\end{proof}

The above lemma  shows that  the   moment of inertia of $ \omega_{\varepsilon} $ will tend to $ dr_*^2/2  $ and on the set $ \Omega\cap\{x\mid |x|\ge r_*+\eta\} $ for any $ \eta>0 $, the circulation of $ \omega_{\varepsilon} $ tends to 0 as $ \varepsilon\to 0^+ $. Based on this estimate, we can then get limits of $ \bar{Q}_\varepsilon. $
\begin{lemma}\label{le6-4}
	$\lim\limits_{\varepsilon\to 0^+}\bar{Q}_\varepsilon=r_*$.
\end{lemma}

\begin{proof}
	Clearly, $ \liminf\limits_{\varepsilon\to 0^+}\bar{Q}_\varepsilon\geq\lim\limits_{\varepsilon\to 0^+}\bar{P}_\varepsilon =r_*  $.  Note that  $$ \sqrt{\det K_H(x_\varepsilon)}^{-1}\leq \sqrt{\det K_H((R^*,0))}^{-1}<+\infty. $$ Taking $x_\varepsilon =Q_\varepsilon $ in \eqref{214}, we obtain
	\begin{equation*}
	\begin{split}
&\frac{\sqrt{\det K_H((R^*,0))}^{-1}}{2\pi}\liminf_{\varepsilon\to 0^+}\int\limits_{\Omega\cap\{|T_{Q_\varepsilon}(y-Q_\varepsilon)|\leq\varepsilon^\gamma\}}\omega_\varepsilon dy\\
&\ \ \ge  Y((r_*,0))+ \frac{\alpha}{2}\liminf_{\varepsilon\to 0^+}(\bar{Q}_\varepsilon)^2+\frac{ \alpha r_*^2}{2}-\frac{\sqrt{\det K_H((R^*,0))}^{-1}d\gamma}{2\pi}  \\
&\ \ \ge	 \frac{d \sqrt{\det K_H((r_*,0))}^{-1}}{2\pi}-\frac{\sqrt{\det K_H((R^*,0))}^{-1}d\gamma}{2\pi}.
	\end{split}
	\end{equation*}
Hence
	\begin{equation*}
\liminf_{\varepsilon\to 0^+}\int\limits_{\Omega\cap\{|T_{Q_\varepsilon}(y-Q_\varepsilon)|\leq\varepsilon^\gamma\}}\omega_\varepsilon dy \ge \frac{d \sqrt{\det K_H((r_*,0))}^{-1}}{\sqrt{\det K_H((R^*,0))}^{-1}}-d\gamma.
	\end{equation*}
So $ \liminf\limits_{\varepsilon\to 0^+}\int\limits_{\Omega\cap\{|T_{Q_\varepsilon}(y-Q_\varepsilon)|\leq\varepsilon^\gamma\}}\omega_\varepsilon dy \ge C_0>0 $ if we choose $ \gamma $  such that $ 0<\gamma<\frac{ \sqrt{\det K_H((r_*,0))}^{-1}}{\sqrt{\det K_H((R^*,0))}^{-1}}. $	

We claim that $ \limsup\limits_{\varepsilon\to 0^+}\bar{Q}_\varepsilon\leq r_*. $ If not, there exists a sequence $ \{\bar{Q}_{\varepsilon,n}\} $ (still denoted by $ \{\bar{Q}_{\varepsilon}\} $ for simplicity) such that $ \lim\limits_{\varepsilon\to 0^+}\bar{Q}_\varepsilon> r_* $. Note that $ T_{Q_\varepsilon} $ is a positive-definite matrix whose eigenvalues have positive upper and lower bounds uniformly about $ \varepsilon. $ Then we find  $ \eta_0>0 $, such that $ \{|T_{Q_\varepsilon}(y-Q_\varepsilon)|\leq\varepsilon^\gamma\}\subset \{ |y|\ge r_*+\eta_0\} $ for $ \varepsilon $ sufficiently small. Thus $  \int\limits_{\{ |y|\ge r_*+\eta_0\}}\omega_\varepsilon dy \ge C_0, $ which contradicts to \eqref{112}. To conclude, we get 	$\lim\limits_{\varepsilon\to 0^+}\bar{Q}_\varepsilon=r_*$.
\end{proof}

From Lemma \ref{le6-1} and Lemma \ref{le6-4}, we know that $ supp(\omega_\varepsilon) $ locates in a narrow annular domain whose centerline is $ \partial B_{r_*}(0) $. That is, for any $ x\in supp(\omega_\varepsilon) $
\begin{equation*}
\lim_{\varepsilon\to 0^+}|x|=r_*.
\end{equation*}
The following lemma shows that the diameter of the support set of $ \omega_\varepsilon $  is the order of $ \varepsilon^{\gamma} $ for any $ \gamma\in(0, 1) $.
\begin{lemma}\label{diam1}$\mathrm{[Diameter~of}$ $ \text{supp}(\omega_\varepsilon)  $$\mathrm{]}$
For any number $\gamma\in(0, 1)$, there holds for $\varepsilon$   small enough
$$ diam\big({supp}(\omega_\varepsilon)\big) \le 2 \varepsilon^{\gamma}. $$
\end{lemma}

\begin{proof}
Since $\int_{\Omega}\omega_\varepsilon dx=d, $ it suffices to prove that
	\begin{equation}\label{conc}
	\int\limits_{B_{\varepsilon^\gamma}(x_\varepsilon)}\omega_\varepsilon dx>d/2, \ \ \forall\, x_\varepsilon \in {supp}(\omega_\varepsilon).
	\end{equation}
For any $ x_\varepsilon\in {supp}(\omega_\varepsilon) $, by $\eqref{214}$  we get
\begin{equation}\label{4-6}
\begin{split}
\int\limits_{\Omega\cap\{|T_{x_\varepsilon}(y-x_\varepsilon)|\leq\varepsilon^\gamma\}}\omega_\varepsilon dy
 \ge  &2\pi\sqrt{\det K_H(x_\varepsilon)}\bigg[Y((r_*,0))+\frac{\alpha}{2d}\left \{\int_\Omega |x|^2 \omega_\varepsilon dx+d|x_\varepsilon|^2\right \}\\
&-\frac{\sqrt{\det K_H((R^*,0))}^{-1}d\gamma}{2\pi}-\frac{C}{\ln\frac{1}{\varepsilon}}\bigg].
\end{split}
\end{equation}	
From Lemmas \ref{le6-1} and \ref{le6-4}, we know that $|x_\varepsilon|\to r_*$  as $\varepsilon \to 0^+$. Taking this into \eqref{4-6}, we obtain
\begin{equation}\label{216}
	\liminf_{\varepsilon\to 0^+}\int\limits_{\{|T_{x_\varepsilon}(y-x_\varepsilon)|\leq\varepsilon^\gamma\}}\omega_\varepsilon dy \ge d-\frac{\sqrt{\det K_H((r_*,0))} d\gamma}{\sqrt{\det K_H((R^*,0))}}.
	\end{equation}
From the definition of $ T_{x_\varepsilon} $, one computes directly that
\begin{equation*}
\{y  \mid T_{x_\varepsilon}(y-x_\varepsilon)|\leq\varepsilon^\gamma\}\subset B_{\varepsilon^\gamma}(x_\varepsilon).
\end{equation*}
Thus we have
\begin{equation*}
	\liminf_{\varepsilon\to 0^+}\int\limits_{B_{\varepsilon^\gamma}(x_\varepsilon)}\omega_\varepsilon dy \ge d-\frac{\sqrt{\det K_H((r_*,0))} d\gamma}{\sqrt{\det K_H((R^*,0))}},
\end{equation*}
which implies \eqref{conc} for all small $\gamma$ such that $d-\frac{\sqrt{\det K_H((r_*,0))} d\gamma}{\sqrt{\det K_H((R^*,0))}}>\frac{d}{2}$. Thus we conclude that $ diam\big({supp}(\omega_\varepsilon)\big) \le 2 \varepsilon^{\gamma} $ for small $ \gamma $, from which we deduce that
\begin{equation}\label{4-7}
diam\big({supp}(\omega_\varepsilon)\big)\le \frac{C}{\ln\frac{1}{\varepsilon}}
\end{equation}
provided $\varepsilon$ is small enough.

Now we go back to $\eqref{212}$ and improve estimates of $ \int\limits_{\Omega\cap\{|T_{x_\varepsilon}(y-x_\varepsilon)|>\varepsilon^\gamma\}}G_{K_H}(x_\varepsilon,y)\omega_\varepsilon(y)dy $. From \eqref{4-7}, we get for any $ \gamma\in(0,1) $
\begin{equation*}
\begin{split}
&\int\limits_{\Omega\cap\{|T_{x_\varepsilon}(y-x_\varepsilon)|>\varepsilon^\gamma\}}G_{K_H}(x_\varepsilon,y)\omega_\varepsilon(y)dy\\
&\ \ \leq\int\limits_{\Omega\cap\{|T_{x_\varepsilon}(y-x_\varepsilon)|>\varepsilon^\gamma\}}\frac{\sqrt{\det K_H(x_\varepsilon)}^{-1}+\sqrt{\det K_H(y)}^{-1}}{2}\Gamma\left (\frac{T_{x_\varepsilon}+T_y}{2}\left (x_\varepsilon-y\right )\right )\omega_\varepsilon(y)dy+C\\
&\ \ \le	 \frac{\sqrt{\det K_H(x_\varepsilon)}^{-1}\gamma}{2\pi}\ln\frac{1}{\varepsilon} \int\limits_{\Omega\cap\{|T_{x_\varepsilon}(y-x_\varepsilon)|>\varepsilon^\gamma\}}\omega_\varepsilon dy+C.
\end{split}improve
\end{equation*}
Repeating the proof of \eqref{214}, we can improve $\eqref{214}$ as follows
\begin{equation*}
	\begin{split}
& \frac{\sqrt{\det K_H(x_\varepsilon)}^{-1}}{2\pi}\int\limits_{\Omega\cap\{|T_{x_\varepsilon}(y-x_\varepsilon)|\leq\varepsilon^\gamma\}}\omega_\varepsilon dy+\frac{\sqrt{\det K_H(x_\varepsilon)}^{-1}\gamma}{2\pi} \int\limits_{\Omega\cap\{|T_{x_\varepsilon}(y-x_\varepsilon)|>\varepsilon^\gamma\}}\omega_\varepsilon dy\\
&\ \ \ge   Y((r_*,0))+\frac{\alpha}{2d}\left \{\int_\Omega |x|^2 \omega_\varepsilon dx+d|x_\varepsilon|^2\right \}-\frac{C}{\ln\frac{1}{\varepsilon}},
\end{split}
\end{equation*}
which implies that
\begin{equation}\label{4-8}
\begin{split}
&(1-\gamma)\int\limits_{\Omega\cap\{|T_{x_\varepsilon}(y-x_\varepsilon)|\leq\varepsilon^\gamma\}}\omega_\varepsilon dy\\
&\ \ \geq 2\pi\sqrt{\det K_H(x_\varepsilon)}\left [ Y((r_*,0))+\frac{\alpha}{2d}\left \{\int_\Omega |x|^2 \omega_\varepsilon dx+d|x_\varepsilon|^2\right \}-\frac{C}{ \ln\frac{1}{\varepsilon}}\right ]-d\gamma.
\end{split}
\end{equation}
Taking the limit inferior to \eqref{4-8}, we get
\begin{equation*}
	\liminf_{\varepsilon\to 0^+}\int\limits_{\Omega\cap\{|T_{x_\varepsilon}(y-x_\varepsilon)|\leq\varepsilon^\gamma\}}\omega_\varepsilon dy \ge \frac{1}{1-\gamma}\left [2\pi\sqrt{\det K_H((r_*,0))}(Y((r_*,0))+\alpha r_*^2)-d\gamma\right ]=d,
	\end{equation*}
	which implies \eqref{conc} for all $ \gamma\in (0,1) $. The proof is thus complete.
\end{proof}

A direct consequence of  Lemma \ref{diam1} is the following estimates of $ diam(supp(\omega_\varepsilon)) $.
\begin{corollary}\label{co1}
	There holds
	\begin{equation}\label{217}
	\lim_{\varepsilon\to0^+}\frac{\ln diam\big(supp(\omega_\varepsilon)\big)}{\ln{\varepsilon}}=1.
	\end{equation}
\end{corollary}
\begin{proof}
Since $\int_\Omega \omega_\varepsilon dx =d$ and $ 0\leq\omega_\varepsilon\leq 1/{\varepsilon^2}$, from the isoperimetric inequality we know that there exists $ r_1>0 $ independent of $ \varepsilon $, such that $$ diam(supp(\omega_\varepsilon))\geq r_1\varepsilon. $$ Thus we have
	\begin{equation}\label{10-1}
	\limsup_{\varepsilon\to 0^+}\frac{\ln  diam ( {supp}(\omega_\varepsilon))}{\ln \varepsilon}\le 1.
	\end{equation}
	On the other hand, by Lemma \ref{diam1}, we obtain
	\begin{equation}\label{10-2}
	\liminf_{\varepsilon\to 0^+} \frac{\ln  diam(supp(\omega_\varepsilon))}{\ln \varepsilon}\ge \gamma  \ \ \ \  \forall\, \gamma\in(0,1).
	\end{equation}
	Combining \eqref{10-1} and \eqref{10-2}, we get the result.
\end{proof}

To sum up, we have obtained the asymptotic behavior of $ \omega_\varepsilon $ as follows.
\begin{proposition}\label{le6}
$\mathrm{[Diameter~and~location~of}$ $ \text{supp}(\omega_\varepsilon) $$\mathrm{]}$
	For any $ \gamma\in (0,1), $ there holds
	\begin{equation*}
	diam[\text{supp}(\omega_\varepsilon)] \le 2\varepsilon^{\gamma}
	\end{equation*}
	provided $\varepsilon$ is small enough. Moreover,
	\begin{equation*}
	\lim_{\varepsilon \to 0^+}dist\left( \text{supp}(\omega_\varepsilon), \partial B_{r_*}(0) \right)=0,
	\end{equation*}
		\begin{equation*}
	\lim_{\varepsilon\to0^+}\frac{\ln diam\big(supp(\omega_\varepsilon)\big)}{\ln{\varepsilon}}=1.
	\end{equation*}
\end{proposition}

Furthermore, we can improve Proposition \ref{le6} by showing that the diameter of $ supp(\omega_\varepsilon) $ is of the order of $ \varepsilon $, rather than $ \varepsilon^\gamma $ for any $ \gamma\in(0,1) $.
First,   we get the following optimal asymptotic expansions of the energy $ \mathcal{E}_\varepsilon(\omega_\varepsilon) $ and Lagrange multiplier $ \mu^\varepsilon $.
\begin{lemma}\label{le10}
As $\varepsilon \to 0^+$, there holds
	\begin{align}
	\label{4-9} \mathcal{E}_\varepsilon(\omega_\varepsilon) & =\left(\frac{d^2}{4\pi\sqrt{\det K_H((r_*,0))}}-\frac{d \alpha r_*^2}{2}\right)\ln{\frac{1}{\varepsilon}}+O(1), \\
	\label{4-10}  \mu^\varepsilon & =\left(\frac{d}{2\pi\sqrt{\det K_H((r_*,0))}}-\frac{  \alpha r_*^2}{2}\right)\ln{\frac{1}{\varepsilon}}+O(1).
	\end{align}
\end{lemma}
\begin{proof}
We first prove $\eqref{4-9}$. According to Lemma $\ref{diam1}$, there holds
	\begin{equation*}
	supp\,(\omega_\varepsilon)\subseteq B_{2\varepsilon^{\frac{1}{2}}}( P_\varepsilon)
	\end{equation*}
for all sufficiently small $\varepsilon$. Hence from Proposition \ref{expe of G_H}  we get
	\begin{equation*}
	\begin{split}
	&\int_\Omega\omega_\varepsilon \mathcal{G}_{K_H} \omega_\varepsilon dx \\
	&\ \ \le \iint\limits_{\Omega\times\Omega}\frac{\sqrt{\det K_H(x)}^{-1}+\sqrt{\det K_H(y)}^{-1}}{2}\Gamma\left (\frac{T_{x}+T_y}{2}\left (x-y\right )\right )\omega_\varepsilon(x)\omega_\varepsilon(y)dxdy+C\\
	&\ \ \leq\frac{\sqrt{\det K_H(P_\varepsilon)}^{-1}+C\varepsilon^{\frac{1}{2}}}{2\pi}\iint\limits_{\Omega\times\Omega}\left (\ln\frac{1}{|T_{P_\varepsilon}(x-y)|}+C\varepsilon^{\frac{1}{2}}\right )\omega_\varepsilon(x)\omega_\varepsilon(y)dxdy+C\\
	&\ \ \leq \frac{\sqrt{\det K_H(P_\varepsilon)}^{-1}}{2\pi}\iint\limits_{\Omega\times\Omega}\ln\frac{1}{|T_{P_\varepsilon}(x-y)|}\omega_\varepsilon(x)\omega_\varepsilon(y)dxdy+C\\
	&\ \ =\frac{\sqrt{\det K_H(P_\varepsilon)}^{-1}}{2\pi}\iint\ln\frac{1}{|x'-y'|}\omega_\varepsilon(T_{P_\varepsilon}^{-1}x')\omega_\varepsilon(T_{P_\varepsilon}^{-1}y')\det K_H(P_\varepsilon)dx'dy'+C.
	\end{split}
	\end{equation*}
By the Riesz's rearrangement inequality,
	\begin{equation*}
	\begin{split}
\iint\ln\frac{1}{|x'-y'|}\omega_\varepsilon(T_{P_\varepsilon}^{-1}x')\omega_\varepsilon(T_{P_\varepsilon}^{-1}y')dx'dy'\le& \ln\frac{1}{\varepsilon}\left \{\int\omega_\varepsilon(T_{P_\varepsilon}^{-1}x')dx'\right \}^2+C\\
=&	d^2\det K_H(P_\varepsilon)^{-1}\ln\frac{1}{\varepsilon}+C.
	\end{split}
	\end{equation*}
Thus we get
	\begin{equation*}
	\int_\Omega\omega_\varepsilon \mathcal{G}_{K_H} \omega_\varepsilon dx \le \frac{d^2\sqrt{\det K_H(P_\varepsilon)}^{-1}}{2\pi}\ln\frac{1}{\varepsilon} +C,
	\end{equation*}
from which we conclude that
	\begin{equation*}
	\begin{split}
\mathcal{E}_\varepsilon(\omega_\varepsilon)=&\frac{1}{2}\int_{\Omega}\omega_\varepsilon\mathcal{G}_{K_H}\omega_\varepsilon dx-\frac{\alpha }{2}\log\frac{1}{\varepsilon}\int_{\Omega}|x|^2\omega_\varepsilon  dx \\
\le& \left(\frac{d^2\sqrt{\det K_H(P_\varepsilon)}^{-1}}{4\pi}-\frac{d \alpha|P_\varepsilon|^2}{2}\right)\ln{\frac{1}{\varepsilon}}+C\\
\le& \left(\frac{d^2}{4\pi\sqrt{\det K_H((r_*,0))}}-\frac{d \alpha r_*^2}{2}\right)\ln{\frac{1}{\varepsilon}}+C.
	\end{split}
	\end{equation*}
Combining the above estimate with Lemma $\ref{lowbdd of E}$, we get $\eqref{4-9}$.  $\eqref{4-10}$ clearly follows from  $\eqref{4-9}$ and Lemma \ref{lowbdd of mu}.
\end{proof}

Based on   Lemma \ref{le10}, we can further improve results of Proposition \ref{le6} that the diameter of the support set of $ \omega_\varepsilon $ is the order  of $ \varepsilon $.
\begin{lemma}\label{sharp}
	
	$\mathrm{[Refined~estimate~of}$  $diam(\text{supp}(\omega_\varepsilon)) $$\mathrm{]}$
There exist  $r_1, r_2>0$ independent of $\varepsilon$ such that for $ \varepsilon $ sufficiently  small,
	\begin{equation*}
	r_1 \varepsilon\leq diam(\text{supp}(\omega_\varepsilon)) \le r_2 \varepsilon.
	\end{equation*}
\end{lemma}
\begin{proof}
In the proof of Corollary \ref{co1}, we have shown that $ diam(  {supp}(\omega_\varepsilon))\geq r_1 \varepsilon$ for some $ r_1>0 $. It suffices to prove $ diam({supp}(\omega_\varepsilon)) \le r_2 \varepsilon $ for some $r_2>0$.
	
Recalling that for any $  x_\varepsilon \in diam( {supp}(\omega_\varepsilon))  $,
	\begin{equation}\label{221-1}
	\mathcal{G}_{K_H}\omega_\varepsilon(x_\varepsilon)-\frac{\alpha|x_\varepsilon|^2}{2}\ln{\frac{1}{\varepsilon}}\ge \mu^{\varepsilon}.
	\end{equation}
According to Proposition $\ref{le6}$, there holds
	\begin{equation}\label{4-11}
	supp(\omega_\varepsilon)\subseteq B_{2\varepsilon^{\frac{1}{2}}} ( x_\varepsilon ).
	\end{equation}
Set $R>1$ to be determined. On the one hand,
	\begin{equation*}
	\begin{split}
	\mathcal{G}_{K_H}\omega_\varepsilon(x_\varepsilon)
	&=\left (\int_{\Omega\cap\{|T_{x_\varepsilon}(y-x_\varepsilon)|> R\varepsilon\}}+\int_{\Omega\cap\{|T_{x_\varepsilon}(y-x_\varepsilon)|\leq R\varepsilon\}}\right )G_{K_H}(x_\varepsilon,y)\omega_\varepsilon(y)dy\\
	&=:B_1+B_2.
	\end{split}
	\end{equation*}
	By \eqref{4-11}, we get
	\begin{equation}\label{4-12}
	\begin{split}
	B_1 &=\int_{\Omega\cap\{|T_{x_\varepsilon}(y-x_\varepsilon)|> R\varepsilon\}}G_{K_H}(x_\varepsilon,y)\omega_\varepsilon(y)dy\\
	&\le \frac{\sqrt{\det K_H(x_\varepsilon)}^{-1}+O\left (\varepsilon^{\frac{1}{2}}\right )}{2\pi}\ln{\frac{1}{R\varepsilon}}\int\limits_{\Omega\cap\{|T_{x_\varepsilon}(y-x_\varepsilon)|> R\varepsilon\}}\omega_\varepsilon dy+C\\
	&\le \frac{\sqrt{\det K_H(x_\varepsilon)}^{-1}}{2\pi}\ln{\frac{1}{R\varepsilon}}\int\limits_{\Omega\cap\{|T_{x_\varepsilon}(y-x_\varepsilon)|> R\varepsilon\}}\omega_\varepsilon dy+C,
	\end{split}
	\end{equation}
	and
	\begin{equation}\label{4-13}
	\begin{split}
	B_2 &=\int_{\Omega\cap\{|T_{x_\varepsilon}(y-x_\varepsilon)|\leq R\varepsilon\}}G_{K_H}(x_\varepsilon,y)\omega_\varepsilon(y)dy\\
	&\le \frac{\sqrt{\det K_H(x_\varepsilon)}^{-1}+O\left (\varepsilon^\frac{1}{2}\right )}{2\pi}\int\limits_{\Omega\cap\{|T_{x_\varepsilon}(y-x_\varepsilon)|\leq R\varepsilon\}}\left (\ln\frac{1}{|T_{x_\varepsilon}(y-x_\varepsilon)|}+O\left (\varepsilon^{\frac{1}{2}}\right )\right )\omega_\varepsilon(y)dy+C\\
	&\le \frac{\sqrt{\det K_H(x_\varepsilon)}^{-1}+O\left (\varepsilon^\frac{1}{2}\right )}{2\pi}\int\limits_{ |z|\leq R\varepsilon }\left (\ln\frac{1}{|z|}+O\left (\varepsilon^{\frac{1}{2}}\right )\right )\omega_\varepsilon\left (T_{x_\varepsilon}^{-1}z+x_\varepsilon\right )\sqrt{\det K_H(x_\varepsilon)}dz +C\\
&\le \frac{\sqrt{\det K_H(x_\varepsilon)}^{-1}+O\left (\varepsilon^\frac{1}{2}\right )}{2\pi}\ln\frac{1}{\varepsilon}\int\limits_{ |z|\leq R\varepsilon }\omega_\varepsilon\left (T_{x_\varepsilon}^{-1}z+x_\varepsilon\right )\sqrt{\det K_H(x_\varepsilon)}dz +C\\
	& \le \frac{\sqrt{\det K_H(x_\varepsilon)}^{-1}}{2\pi}\ln\frac{1}{\varepsilon}\int\limits_{\Omega\cap\{|T_{x_\varepsilon}(y-x_\varepsilon)|\leq R\varepsilon\}}\omega_\varepsilon dy+C.
	\end{split}
	\end{equation}
	Taking \eqref{4-12} and \eqref{4-13} into \eqref{221-1}, we get
	\begin{equation}\label{222-1}
	\begin{split}
&\frac{\sqrt{\det K_H(x_\varepsilon)}^{-1}}{2\pi}\ln{\frac{1}{R\varepsilon}}\int\limits_{\{|T_{x_\varepsilon}(y-x_\varepsilon)|> R\varepsilon\}}\omega_\varepsilon dy+\frac{\sqrt{\det K_H(x_\varepsilon)}^{-1}}{2\pi}\ln\frac{1}{\varepsilon}\int\limits_{\{|T_{x_\varepsilon}(y-x_\varepsilon)|\leq R\varepsilon\}}\omega_\varepsilon dy\\
&\ \ \ge \frac{\alpha|x_\varepsilon|^2}{2}\ln{\frac{1}{\varepsilon}}+\mu^{\varepsilon}-C .
	\end{split}
	\end{equation}
On the other hand, by Lemma $\ref{le10}$, one has
	\begin{equation}\label{223}
	\mu^\varepsilon\ge\left(\frac{d}{2\pi\sqrt{\det K_H(x_\varepsilon)}}-\frac{  \alpha |x_\varepsilon|^2}{2}\right)\ln{\frac{1}{\varepsilon}}-C.
	\end{equation}
Combining $\eqref{222-1}$ and $\eqref{223}$, we obtain
	\begin{equation*}
	\frac{d}{2\pi}\ln\frac{1}{\varepsilon} \le \frac{1}{2\pi}\ln{\frac{1}{R\varepsilon}}\int\limits_{\{|T_{x_\varepsilon}(y-x_\varepsilon)|> R\varepsilon\}}\omega_\varepsilon dy+\frac{1}{2\pi}\ln\frac{1}{\varepsilon}\int\limits_{\{|T_{x_\varepsilon}(y-x_\varepsilon)|\leq R\varepsilon\}}\omega_\varepsilon dy+C,
	\end{equation*}
which implies that
	\begin{equation*}
\int\limits_{\{|T_{x_\varepsilon}(y-x_\varepsilon)|\leq R\varepsilon\}}\omega_\varepsilon dy\ge d\left (1-\frac{C}{\ln R}\right ).
	\end{equation*}
	Taking $R>1$ such that $C(\ln R)^{-1}<1/2$, we obtain
	\begin{equation*}
	\int\limits_{\Omega\cap{B_{R\varepsilon}(x_\varepsilon)}}\omega_\varepsilon dy\ge\int\limits_{\{|T_{x_\varepsilon}(y-x_\varepsilon)|\leq R\varepsilon\}}\omega_\varepsilon dy> \frac{\kappa}{2}.
	\end{equation*}
	Taking $r_2=2R$, we complete the proof of Lemma \ref{sharp}.	
\end{proof}
\subsection{Asymptotic shape of $ supp(\omega_\varepsilon) $}
Having estimated the limiting location and diameter of the support set of $ \omega_\varepsilon  $, now we consider the asymptotic shape of $ supp(\omega_\varepsilon) $.

To this end, let us define the center of $supp(\omega_\varepsilon)$
\begin{equation*}
X^\varepsilon=\frac{\int_\Omega x\omega_\varepsilon(x)dx}{\int_\Omega \omega_\varepsilon(x)dx}.
\end{equation*}
The scaled function of $ \omega_\varepsilon $ is denoted by
\begin{equation*}
 \eta^\varepsilon(x):=\varepsilon^2 \omega_\varepsilon(X^\varepsilon+\varepsilon x),\ \ \ x\in\mathbb{R}^2.
\end{equation*}
From Proposition \ref{le6} and Lemma \ref{sharp}, we know that there exists $ x_*\in\Omega $ with $ |x_*|= r_* $ such that $ |X^\varepsilon-x_*|\to 0$  as $ \varepsilon\to0 $ and $  supp(\eta^\varepsilon)\subset B_{r_2}(0) $.

Different from the 2D Euler and 3D axisymmetric Euler cases, we will prove that the limiting shape of the support set $ \eta^\varepsilon $ is an ellipse, rather than a disc.  To this end, we denote
\begin{equation}\label{4-15}
\zeta^\varepsilon(x):=\eta^\varepsilon\left (T_{X^\varepsilon}^{-1}x\right ),\ \ \ x\in\mathbb{R}^2.
\end{equation}
From Proposition \ref{exist of max}, we know that $ \zeta^\varepsilon=\textbf{1}_{K_\varepsilon} $ for some $ K_\varepsilon\subseteq B_{r_2}(0)\subseteq\mathbb{R}^2 $ and
\begin{equation}\label{4-14}
|K_\varepsilon|=\int\zeta^\varepsilon dx=d\sqrt{\det K_H(X^\varepsilon)}^{-1}.
\end{equation}
By the uniform  boundedness of $ \zeta^\varepsilon $  in $ L^p(B_{r_2}(0)) $ for any $ p\in(1,+\infty] $,  up to a subsequence (still denoted by $ \zeta^\varepsilon $), we may assume that
\begin{equation*}
\zeta^\varepsilon \to \zeta
\end{equation*}
in $ L^{p} $ weak topology for $ p\in (1,+\infty) $ and $ L^{ \infty} $ weak star topology for some  $ \zeta\in L^{\infty}(B_{r_2}(0))  $ as $ \varepsilon\to0. $  It is not hard to  prove that  $ 0\leq \zeta\leq 1 $ and $ supp(\zeta)\subset B_{r_2}(0) $.

Denote $ (\zeta^\varepsilon)^\bigtriangleup $ the symmetric radially nonincreasing
Lebesgue-rearrangement of $ \zeta^\varepsilon $ centered on 0. Then one computes directly that
\begin{equation*}
(\zeta^\varepsilon)^\bigtriangleup \to \check{\zeta}:=\textbf{1}_{B_{c^*}(0)},
\end{equation*}
in $ L^p $ topology for $ p\in(1,+\infty) $  as $ \varepsilon\to0 $, where $ \pi(c^*)^2=d\sqrt{\det K_H((r_*,0))}^{-1}. $

 The following result shows the asymptotic shape of $\zeta^\varepsilon$.

\begin{lemma}\label{le11}
There holds
\begin{equation*}
\zeta^\varepsilon \to  \check{\zeta}=\textbf{1}_{B_{c^*}(0)}
\end{equation*}
in $ L^p $ topology for any $ p\in(1,+\infty) $  as $ \varepsilon\to0 $.
\end{lemma}
\begin{proof}
On the one hand, by  the Riesz' rearrangement inequality, we have
	\begin{equation*}
	\iint\limits_{B_{r_2}(0)\times B_{r_2}(0)}\ln\frac{1}{|x-x'|}\zeta^\varepsilon(x)\zeta^\varepsilon(x')dxdx'
	\le \iint\limits_{B_{r_2}(0)\times B_{r_2}(0)}\ln\frac{1}{|x-x'|}(\zeta^\varepsilon)^\bigtriangleup(x)(\zeta^\varepsilon)^\bigtriangleup(x')dxdx'.
	\end{equation*}
	Hence passing $ \varepsilon $ to limit, one has
	\begin{equation}\label{226}
	\iint\limits_{B_{r_2}(0)\times B_{r_2}(0)}\ln\frac{1}{|x-x'|}\zeta(x)\zeta(x')dxdx'\le \iint\limits_{B_{r_2}(0)\times B_{r_2}(0)}\ln\frac{1}{|x-x'|}\check{\zeta}(x)\check{\zeta}(x')dxdx'.
	\end{equation}
On the other hand, let $\check{\omega}_\varepsilon\in\mathcal{M}_\varepsilon$ be defined as
	\[
\check{\omega}_\varepsilon(x)=\left\{
	\begin{array}{lll}
	\frac{1}{\varepsilon^2}(\zeta^\varepsilon)^\bigtriangleup\left (\frac{T_{X^\varepsilon}(x-X^\varepsilon)}{\varepsilon}\right ) & \ \    \text{if}\ \  |T_{X^\varepsilon}(x-X^\varepsilon)|\leq r_2\varepsilon, \\
	0                  & \ \    \text{otherwise}.
	\end{array}
	\right.
	\]
Then $\check{\omega}_\varepsilon\in\mathcal{M}_\varepsilon$ for $ \varepsilon $ small enough. One computes directly that
	\begin{equation*}
	\begin{split}
\mathcal{E}_\varepsilon(\omega_\varepsilon) =&\frac{d^2\sqrt{\det K_H(X^\varepsilon)}^{-1}}{4\pi}\ln\frac{1}{\varepsilon}+\frac{\sqrt{\det K_H(X^\varepsilon)}}{4\pi}\iint\limits_{B_{r_2}(0)\times B_{r_2}(0)}\ln\frac{1}{|x-x'|}\zeta^\varepsilon(x)\zeta^\varepsilon(x')dxdx' \\
	& +\frac{1}{2}S_{K_H}(X^\varepsilon,X^\varepsilon)d^2-\frac{\alpha(X_\varepsilon)^2d}{2}\ln\frac{1}{\varepsilon}+O(\varepsilon^\gamma),
	\end{split}
	\end{equation*}
	and
	\begin{equation*}
\begin{split}
\mathcal{E}_\varepsilon(\check{\omega}_\varepsilon) =&\frac{d^2\sqrt{\det K_H(X^\varepsilon)}^{-1}}{4\pi}\ln\frac{1}{\varepsilon}+\frac{\sqrt{\det K_H(X^\varepsilon)}}{4\pi}\iint\limits_{B_{r_2}(0)\times B_{r_2}(0)}\ln\frac{1}{|x-x'|}(\zeta^\varepsilon)^\bigtriangleup(x)(\zeta^\varepsilon)^\bigtriangleup(x')dxdx' \\
& +\frac{1}{2}S_{K_H}(X^\varepsilon,X^\varepsilon)d^2-\frac{\alpha(X_\varepsilon)^2d}{2}\ln\frac{1}{\varepsilon}+O(\varepsilon^\gamma),
\end{split}
\end{equation*}
for some $ \gamma\in(0,1). $ Recalling that $\mathcal{E}_\varepsilon(\omega_\varepsilon)\geq \mathcal{E}_\varepsilon(\check{\omega}_\varepsilon) $, we have
	\begin{equation*}
\iint\limits_{B_{r_2}(0)\times B_{r_2}(0)}\ln\frac{1}{|x-x'|}\zeta^\varepsilon(x)\zeta^\varepsilon(x')dxdx'
\geq \iint\limits_{B_{r_2}(0)\times B_{r_2}(0)}\ln\frac{1}{|x-x'|}(\zeta^\varepsilon)^\bigtriangleup(x)(\zeta^\varepsilon)^\bigtriangleup(x')dxdx'+O(\varepsilon^\gamma).
\end{equation*}	
Thus we conclude that
	\begin{equation*}
	\iint\limits_{B_{r_2}(0)\times B_{r_2}(0)}\ln\frac{1}{|x-x'|}\zeta(x)\zeta(x')dxdx'\ge \iint\limits_{B_{r_2}(0)\times B_{r_2}(0)}\ln\frac{1}{|x-x'|}\check{\zeta}(x)\check{\zeta}(x')dxdx',
	\end{equation*}
which together with $\eqref{226}$ yields to
	\begin{equation*}
	\iint\limits_{B_{r_2}(0)\times B_{r_2}(0)}\ln\frac{1}{|x-x'|}\zeta(x)\zeta(x')dxdx'= \iint\limits_{B_{r_2}(0)\times B_{r_2}(0)}\ln\frac{1}{|x-x'|}\check{\zeta}(x)\check{\zeta}(x')dxdx'.
	\end{equation*}
By lemma 3.2 in \cite{BG}, there exists a translation $\mathcal{T}_0$ in $\mathbb{R}^2$ such that $\mathcal{T}_0\zeta=\check{\zeta}$. Note that $$ \int_{B_{r_2}(0)}x\zeta(x)dx=\int_{B_{r_2}(0)}x\check{\zeta}(x)dx=0. $$
Thus we get $\zeta=\check{\zeta}$  and   $ \zeta^\varepsilon \to  \check{\zeta}=\textbf{1}_{B_{c^*}(0)} $ in $ L^p $ weak topology as $ \varepsilon\to0. $ By the weak lower semi-continuity of $ L^p $ norm and \eqref{4-14}, we get
\begin{equation*}
(\pi (c^*)^2)^{\frac{1}{p}}=||\check{\zeta}||_{L^p}\le \liminf_{\varepsilon\to 0^+}||\zeta^\varepsilon||_{L^p}\le \limsup_{\varepsilon\to 0^+}||\zeta^\varepsilon||_{L^p}\leq (\pi (c^*)^2)^{\frac{1}{p}}.
\end{equation*}
Using the strict convexity of $ L^p $ norm, we finish the proof.
\end{proof}

\begin{remark}
Indeed, we can further prove that the boundary
of $ supp(\zeta^\varepsilon) $ is a $ C^1 $ curve and converges to $ \partial B_{c^*}(0) $ in
$ C^1 $ sense as $ \varepsilon\to0 $, see \cite{T} for instance. Thus from \eqref{4-15}, we know that the limiting shape of $ supp(\eta^\varepsilon) $ is an ellipse $ T^{-1}_{x_*}B_{c^*}(0) $, rather than a disc.
\end{remark}

\section{Proof of Theorem \ref{thm01}}
Having made all the preparation, we are to give proof of Theorem \ref{thm01}. Let $k>0, d>0, R^*>0$  be fixed numbers. For any $ r_*\in (0,R^*) $,  let
\begin{equation*}
\alpha=\frac{d}{4\pi k\sqrt{k^2+r_*^2}}.
\end{equation*}
From Proposition \ref{exist of max}, we know that for any $ \varepsilon \in(0,\min\{1, \sqrt{|\Omega|/d}\} ) $ all the maximizers of $ \mathcal{E}_\varepsilon $ over $ \mathcal{M}_\varepsilon $ is of the form
\begin{equation*}
\omega_{\varepsilon}=\frac{1}{\varepsilon^2}\mathbf{1}_{\{\mathcal{G}_{K_H}\omega_{\varepsilon}-\frac{\alpha |x|^2}{2}\ln\frac{1}{\varepsilon} -\mu^{\varepsilon}>0\}}  \ \ a.e.\  \text{in}\  \Omega,
\end{equation*}
where $\mu^{\varepsilon} $ is a  Lagrange multiplier determined by $\omega_{\varepsilon}$. Moreover, $  \omega_{\varepsilon} $ is a weak solution of \eqref{rot eq2} with $ f_\varepsilon(t)=\frac{1}{\varepsilon^2}\textbf{1}_{\{t>\mu^\varepsilon\}} $, which implies that the solution pair $ (\omega_{\varepsilon}, \mathcal{G}_{K_H}\omega_{\varepsilon}) $ satisfies \eqref{rot eq}.

By Proposition \ref{le6} and Lemma \ref{sharp}, the support set $ \bar{A}_\varepsilon $ of the maximizer $ \omega_\varepsilon $ concentrates near $ (r_*,0) $ up to rotation as $ \varepsilon\to0 $ with  $\int_{\Omega}\omega_\varepsilon dx=\frac{1}{\varepsilon^2} \int_{\bar{A}_\varepsilon}\mathbf{1}_{\{\mathcal{G}_{K_H}\omega_{\varepsilon}-\frac{\alpha |x|^2}{2}\ln\frac{1}{\varepsilon} -\mu^{\varepsilon}>0\}}dx=d $. Define for any $ (x_1,x_2,x_3)\in B_{R^*}(0)\times\mathbb{R}, t\in\mathbb{R}  $
\begin{equation*}
\mathbf{w}_\varepsilon(x_1,x_2,x_3,t)= \frac{w_\varepsilon(x_1,x_2,x_3,t)}{k}\overrightarrow{\zeta},
\end{equation*}
where $ w_\varepsilon(x_1,x_2,x_3,t) $ is a helical function satisfying
\begin{equation}\label{503}
w_\varepsilon(x_1,x_2,0,t)=\omega_\varepsilon(\bar{R}_{-\alpha|\ln\varepsilon| t}(x_1,x_2))=\frac{1}{\varepsilon^2}\mathbf{1}_{\{\mathcal{G}_{K_H}\omega_{\varepsilon}(\bar{R}_{-\alpha|\ln\varepsilon| t}(x_1,x_2))-\frac{\alpha |x|^2}{2}\ln\frac{1}{\varepsilon} -\mu^{\varepsilon}>0\}}.
\end{equation}
Direct computations show that $ w_\varepsilon(x_1,x_2,0,t) $ is a rotating vortex patch to \eqref{vor str equ2} and $ \mathbf{w}_\varepsilon $ is a left-handed helical vorticity field of \eqref{Euler eq2}. Moreover, $ w_\varepsilon(x_1,x_2,0,t) $ rotates clockwise around the origin with angular velocity $ \alpha|\ln\varepsilon| $. By \eqref{1001}, the circulation of $ \mathbf{w}_\varepsilon $ satisfies
\begin{equation*}
\iint_{A_\varepsilon}\mathbf{w}_\varepsilon\cdot\mathbf{n}d\sigma=\frac{1}{\varepsilon^2} \int_{\bar{A}_\varepsilon}\mathbf{1}_{\{\mathcal{G}_{K_H}\omega_{\varepsilon}-\frac{\alpha |x|^2}{2}\ln\frac{1}{\varepsilon} -\mu^{\varepsilon}>0\}}dx=d.
\end{equation*}

It suffices to prove that the vorticity field $ \mathbf{w}_\varepsilon $ tends asymptotically to \eqref{1007} in sense of \eqref{1005}.  Define $ P(\tau) $ the intersection point of the curve  parameterized by \eqref{1007} and the $ x_1Ox_2 $ plane. Note that the helix \eqref{1007} corresponds uniquely to the motion of $ P(\tau) $. Taking $ s=\frac{b_1\tau}{k} $ into \eqref{1007}, one computes directly that  $ P(\tau) $ satisfies a 2D point vortex model
\begin{equation*}
\begin{split}
P(\tau)=\bar{R}_{\alpha'\tau}((r_*,0)),
\end{split}
\end{equation*}
where
\begin{equation*}
\alpha'=\frac{1}{\sqrt{k^2+r_*^2}}\left( a_1+\frac{b_1}{k}\right)=\frac{d}{4\pi k\sqrt{k^2+r_*^2}},
\end{equation*}
which is equal to $ \alpha $. Thus by the construction, we  readily check that the support set of $ w_\varepsilon(x_1,x_2,0,|\ln\varepsilon|^{-1}\tau) $ defined by  \eqref{503} concentrates near $ P(\tau) $ as $ \varepsilon\to 0 $, which implies that, the vorticity field $ \mathbf{w}_\varepsilon $ tends asymptotically to \eqref{1007} in sense of \eqref{1005}.
The proof of Theorem \ref{thm01} is thus complete.

\section{Orbital stability}

Now we give proof of Theorem \ref{os}.
To this end, we need three preliminary lemmas first. Using the energy characterization that any element in $ \mathcal{S}_\varepsilon $ is a maximizer of $ \mathcal{E}_\varepsilon $ in $\mathcal{M}_\varepsilon $,  we can obtain the following compactness result.
\begin{lemma}\label{com}
$\mathrm{[Compactness]}$	Let $\{w_n\}$ be a maximizing sequence for $\mathcal{E}_\varepsilon $ in $\mathcal{M}_\varepsilon $, then up to a subsequence there exists $w^\varepsilon\in \mathcal{S}_\varepsilon $ such that as $n\rightarrow+\infty$, $w_n\rightarrow w^\varepsilon$ in $L^p(\Omega)$ for any $p\in[1,+\infty)$.
\end{lemma}
\begin{proof}
Since $\{w_n\}$ is a bounded sequence in $L^\infty(\Omega)$, it suffices to show that $w_n\rightarrow w^\varepsilon$ in $L^2(\Omega)$.
According to the proof of Proposition \ref{exist of max}, for any maximizing sequence $w_n$, there must be a maximizer $w^\varepsilon\in \mathcal{M}_\varepsilon $ such that $w_n\rightarrow w^\varepsilon$ weakly star in $L^\infty(\Omega)$. Thus $w_n\rightarrow w^\varepsilon$ weakly in $L^2(\Omega)$, which implies
	\begin{equation}\label{4-01}
	\|w^\varepsilon\|_{L^2}\leq\liminf_{n\rightarrow+\infty}\|w_n\|_{L^2}.
	\end{equation}
	On the other hand, by Proposition \ref{exist of max}, $w^\varepsilon=\frac{1}{\varepsilon^2} \textbf{1}_{\tilde{A}^\varepsilon}$ with $|\tilde{A}^\varepsilon|=d\varepsilon^2$, which gives
	\begin{equation}\label{4-02}
	\|w^\varepsilon\|_{L^2 }=\frac{1}{\varepsilon^2}|\tilde{A}^\varepsilon|^{1/2}=\frac{\sqrt{d}}{\varepsilon}.
	\end{equation}
	But for any $n$,
	\begin{equation}\label{4-03}
	\|w_n\|_{L^2} \leq\frac{1}{\varepsilon}\left(\int_\Omega|w_n(x)|dx\right)^{1/2}=\frac{\sqrt{d}}{\varepsilon}.
	\end{equation}
	Combining \eqref{4-02} and \eqref{4-03} we obtain
	\begin{equation}\label{4-04}
	\|w^\varepsilon\|_{L^2}\geq\limsup_{n\rightarrow+\infty}\|w_n\|_{L^2}.
	\end{equation}
	Now by \eqref{4-01} and \eqref{4-04} we have
	\begin{equation*}
	\|w^\varepsilon\|_{L^2}=\lim_{n\rightarrow+\infty}\|w_n\|_{L^2}.
	\end{equation*}
	By the uniform convexity of $L^2$ we conclude that $w_n\rightarrow w^\varepsilon$ in $L^2$.

\end{proof}

For planar Euler flows, Burton obtained the linear transport theory of associated vorticity equations, see  lemmas 11 and 12 in \cite{B5}. Since in \eqref{vor str equ2} the velocity field $ \nabla^\perp \mathcal{G}_{K_H}w $ is also divergence free, following Burton's idea, one can get the linear transport theory of vorticity equations to 3D helical Euler flows as follows.
\begin{lemma}\label{dl}
	Let $w(x,t)\in L^\infty_{loc}(\mathbb R; L^p(\Omega))$ with $2\leq p<+\infty$. Let $\mathbf{u}=\nabla^\perp \mathcal{G}_{K_H}w$, $\zeta_0\in L^p(\Omega)$. Then there exists a weak solution $\zeta(x,t)\in L^\infty_{loc}(\mathbb R; L^p(\Omega))\cap C(\mathbb R; L^p(\Omega))$ to the following linear transport equation
	\begin{equation*}\label{lt}
	\begin{cases}
	\partial_t\zeta+\mathbf{u}\cdot\nabla\zeta=0, &t\in\mathbb R,\\
	\zeta(\cdot,0)=\zeta_0.
	\end{cases}
	\end{equation*}
	Here by weak solution we mean
	\begin{equation*}
	\begin{split}
	\int_{\mathbb R}\int_D\partial_t\phi(x,t)\zeta(x,t)+\zeta(x,t)&(\mathbf{u}\cdot\nabla\phi)(x,t)dxdt=0,\,\,\forall\,\,\phi\in C_c^\infty(D\times\mathbb R),\\
	&\lim_{t\rightarrow0}\|\zeta(\cdot,t)-\zeta_0\|_{L^p(D)}=0.
	\end{split}
	\end{equation*}
	Moreover, we have for any $t\in \mathbb R$
	\begin{equation*}\label{re}
	|\{x\in D\mid \zeta(x,t)>a\}|=|\{x\in D\mid\zeta_0(x)>a\}|,\,\,\forall \,\,a\in\mathbb R.
	\end{equation*}
	As a consequence, we have for any $t\in \mathbb R$
	\begin{equation*}
	\|\zeta(\cdot,t)\|_{L^p(D)}=\|\zeta_0\|_{L^p(D)}.
	\end{equation*}
	
\end{lemma}

\begin{proof}
See  lemmas 11 and 12  in \cite{B5}, for example.
\end{proof}

Using the results in \cite{Ben,B5}, one can get the energy and angular momentum conservation of solutions $ \omega $ to the vorticity equation \eqref{vor str eq}.
\begin{proposition}\label{E and M conserv}
Let $ 2\leq  p<\infty $. Let $ \omega(x,t)\in L^\infty(\mathbb R; L^p(\Omega))  $ be a
solution of the vorticity equation \eqref{vor str eq}. Then the kinetic energy $ \mathcal{E} $ defined by \eqref{KE}  and the angular
momentum $ \mathcal{I} $ defined by \eqref{MI}  are conserved along the time.
\end{proposition}
\begin{proof}
See  lemmas 4.6 and 4.8 in \cite{Ben}, for instance.
\end{proof}

Now we are ready to prove Theorem \ref{os}.
\begin{proof}[Proof of Theorem \ref{os}]
For simplicity of notation, we denote
\begin{equation*}
dist_p(\omega_0,\mathcal{S}_\varepsilon)=\inf_{\omega\in\mathcal{S}_\varepsilon}\|\omega_0-\omega\|_{L^p(\Omega)}.
\end{equation*}
We give the proof by contradiction. Suppose that there exists a $\rho_0>0$, $t_n>0,$ $v_0^n\in L^p(\Omega)$ satisfying $dist_p(v_0^n,\mathcal{S}_\varepsilon)\rightarrow0,$ but
	\begin{equation}\label{444}
	dist_p(v^n_{t_n},\mathcal{S}_\varepsilon)>\rho_0.
	\end{equation}
Here $v^n_{t}$ is a weak solution to the vorticity equation with initial $v^n_0.$
Since $2\leq p<+\infty$, it follows from the energy and angular momentum conservation in Proposition \ref{E and M conserv}  that $\{v^n_{t_n}\}$ satisfies
	\begin{equation}\label{4-06}
	\lim_{n\rightarrow+\infty}\mathcal{E}_\varepsilon(v^n_{t_n})=\sup_{\mathcal{M}_\varepsilon }\mathcal{E}_\varepsilon.
	\end{equation}

Since $dist_p(v^n_0,\mathcal{S}_\varepsilon)\rightarrow0$, we can choose $w^n_0\in\mathcal{S}_\varepsilon$ such that as $n\rightarrow+\infty$
	\begin{equation*}\label{4-07}
	\|w^n_0-v^n_0\|_{L^p}\rightarrow0.
	\end{equation*}
	Now for each $n$, let $w^n(x,t)$ be the solution of the following linear transport equation
	\begin{equation*}\label{4-08}
	\begin{cases}
	\partial_tw^n(x,t)+\nabla^\perp \mathcal{G}_{K_H}v^n_t\cdot\nabla w^n(x,t)=0,\\
	w^n(x,0)=w^n_0(x).
	\end{cases}
	\end{equation*}
	By Lemma \ref{dl}, it is clear that $w^n(\cdot,t)\in \mathcal{M}_\varepsilon $ for any $t>0$, and as $n\rightarrow+\infty$
	\begin{equation}\label{4-09}
	\|w^n(\cdot,t_n)-v^n_{t_n}\|_{L^p(D)}=\|w^n_0-v^n_0\|_{L^p(D)}\rightarrow0.
	\end{equation}
	Combining \eqref{4-06} and \eqref{4-09} we obtain
	\begin{equation*}\label{4-010}
	\lim_{n\rightarrow+\infty}\mathcal{E}_\varepsilon(w^n(\cdot,t_n))=\sup_{\mathcal{M}_\varepsilon }\mathcal{E}_\varepsilon.
	\end{equation*}
	Then by Lemma \ref{com} it follows that there exists $w^\varepsilon\in \mathcal{S}_\varepsilon$ such that $  \|w^n(\cdot,t_n)-w^\varepsilon\|_{L^p(\Omega)}\rightarrow0$, which gives
	\begin{equation}\label{4-011}
	dist_p(w^n(\cdot,t_n),\mathcal{S}_\varepsilon)\rightarrow0.
	\end{equation}
	Now \eqref{444}, \eqref{4-09} and \eqref{4-011} together lead to a contradiction. Thus Theorem \ref{os} is proved.
\end{proof}

\subsection*{Acknowledgments:}

\par
D. Cao was supported by NNSF of China (grant No. 11831009). J. Wan was supported by NNSF of China (grant No. 12101045) and  Beijing
Institute of Technology Research Fund Program for Young Scholars (No.3170011182016).

\subsection*{Conflict of interest statement} On behalf of all authors, the corresponding author states that there is no conflict of interest.

\subsection*{Data availability statement} All data generated or analysed during this study are included in this published article  and its supplementary information files.

\end{document}